\documentclass[a4paper,
               10pt,
               pdftex,
               normalheadings,
               headsepline,
               footsepline,
               headinclude,
               footinclude,
               DIV14,
               abstracton]
{scrartcl}

\usepackage[left=2.50cm, right=2.50cm, top=2.00cm, bottom=3.50cm]{geometry}

\usepackage
{
    graphicx,
    amssymb,
    amsmath,
    amsthm,
    xcolor,
    dsfont,
    algpseudocode,
    authblk,
}
\usepackage[ngerman, english]{babel}

\usepackage{enumerate}
\usepackage{algorithm}
\usepackage{algpseudocode}
\usepackage{tabularx}
\usepackage{subcaption} 
\usepackage{booktabs} 
\usepackage{algpseudocode}

\usepackage[colorlinks,
            pdffitwindow=false,
            plainpages=false,
            pdfpagelabels=true,
            pdfpagemode=UseOutlines,
            pdfpagelayout=SinglePage,
            bookmarks=false,
            colorlinks=true,
            hyperfootnotes=false,
            linkcolor=blue,
            citecolor=green!50!black]
{hyperref}

\usepackage{mathtools}
\usepackage{enumerate}
\usepackage{thmtools}
\usepackage{cite}

\usepackage{sectsty}
\sectionfont{\fontsize{10}{10}\selectfont}
\subsectionfont{\fontsize{10}{10}\selectfont}

\newcommand{\R}{\mathbb{R}}

\renewcommand{\H}{\mathcal{H}}
\newcommand{\E}{\mathcal{E}}

\newcommand{\expnumber}[2]{{#1}\mathrm{e}{#2}}

\DeclareMathOperator*{\argmax}{arg\,max}
\DeclareMathOperator*{\argmin}{arg\,min}

\DeclareMathOperator*{\proj}{proj}
\DeclareMathOperator{\co}{co}
\DeclareMathOperator{\st}{s.t.}

\makeatletter
\newcommand{\leqnomode}{\tagsleft@true\let\veqno\@@leqno}
\newcommand{\reqnomode}{\tagsleft@false\let\veqno\@@eqno}
\makeatother

\newtheorem{theorem}{Theorem}[section]
\newtheorem{myprop}[theorem]{Proposition}
\newtheorem{mydef}[theorem]{Definition}
\newtheorem{mycorollary}[theorem]{Corollary}
\newtheorem{mylemma}[theorem]{Lemma}
\newtheorem{myremark}[theorem]{Remark}

\let\OLDthebibliography\thebibliography
\renewcommand\thebibliography[1]{
	\OLDthebibliography{#1}
	\setlength{\parskip}{0pt}
	\setlength{\itemsep}{0pt plus 0.3ex}
}

\begin{document}

\title{Fast Multiobjective Gradient Methods with Nesterov Acceleration via Inertial Gradient-like Systems}
\author[1]{Konstantin Sonntag}
\author[2]{Sebastian Peitz}
\affil[1]{\normalsize Department of Mathematics, Paderborn University, Germany}
\affil[2]{\normalsize Department of Computer Science, Paderborn University, Germany}

\date{}

\maketitle

\begin{abstract}
    We derive efficient algorithms to compute weakly Pareto optimal solutions for smooth, convex and unconstrained multiobjective optimization problems in general Hilbert spaces. To this end, we define a novel inertial gradient-like dynamical system in the multiobjective setting, whose trajectories converge weakly to Pareto optimal solutions. Discretization of this system yields an inertial multiobjective algorithm which generates sequences that converge weakly to Pareto optimal solutions. We employ Nesterov acceleration to define an algorithm with an improved convergence rate compared to the plain multiobjective steepest descent method (Algorithm \ref{algo:ACC_GRAD}).
    A further improvement in terms of efficiency is achieved by avoiding the solution of a quadratic subproblem to compute a common step direction for all objective functions, which is usually required in first order methods. Using a different discretization of our inertial gradient-like dynamical system, we obtain an accelerated multiobjective gradient method that does not require the solution of a subproblem in each step (Algorithm \ref{algo:ACC_GRAD_wo_Q}). While this algorithm does not converge in general, it yields good results on test problems while being faster than standard steepest descent by two to three orders of magnitude.
\end{abstract}
\section{Introduction}
In many applications in industry, economics, medicine or transport, optimizing several criteria is of interest. In the latter, one wants to reach a destination as fast as possible with minimal power consumption. Drug development aims for maximizing efficacy while minimizing side effects. Even these elementary examples share an inherent feature. The different criteria one seeks to optimize are in general contradictory. There is no design choice that is best for all criteria simultaneously. This insight shifts the focus from finding a single optimal solution to a set of optimal compromises -- the Pareto set. Given the Pareto set, a decision maker can select an optimal compromise according to their preferences. In this paper, we derive efficient gradient-based algorithms to compute elements of the Pareto set. Formally, a problem involving multiple criteria can be formalized as an unconstrained multiobjective optimization problem
\leqnomode
\begin{align}
\tag{MOP}
    \min_{x \in \H} \left[ 
    \begin{array}{c}
        f_1(x)  \\
        \vdots \\
        f_m(x)
    \end{array}
    \right],
    \label{eq:MOP1}
\end{align}
where $f_i:\H \to \R$ for $i = 1, \dots, m$ are the objective functions describing the different criteria. Popular approaches to tackle this problem in the differentiable case are first order methods which exploit the smooth structure of the problem while not being computationally demanding compared to higher order methods involving exact or approximated Hessians. 

While in single objective optimization, \emph{accelerated} first order methods are very popular, these methods are not studied sufficiently from a theoretical point of view in the multiobjective setting. A fruitful approach to analyze accelerated gradient methods is to interpret them as discretizations of suitable gradient-like dynamical systems \cite{Su2014}. The analysis of the continuous dynamics is often easier and can later on be transferred to the discrete setting. This perspective is until now not fully taken advantage of in the area of multiobjective optimization. In this paper, we utilize this approach to derive accelerated gradient methods for multiobjective optimization.
To this end, we define and analyze the following novel dynamical gradient-like system
\begin{align}
\label{eq:IMOG'}
\tag{IMOG'}
        \ddot{x}(t) + \alpha\dot{x}(t) + \proj_{C(x(t))}{(-\ddot{x}(t))} = 0,
\end{align}
with $\alpha >0$, $C(x) \coloneqq \co \left(\lbrace \nabla f_i(x) \, : \, i = 1,\dots,m \rbrace \right)$, where $\co (\cdot)$ denotes the convex hull and $\proj_{C(x(t))}{(-\ddot{x}(t))}$ is the projection of $-\ddot{x}(t)$ onto the convex set $C(x(t))$. The system \eqref{eq:IMOG'} is an \emph{inertial multiobjective gradient-like system}.  We choose the designation \eqref{eq:IMOG'} to emphasize its relation to the system 
\begin{align}
\label{eq:IMOG2}
\tag{IMOG}
    m\ddot{x}(t) + \gamma \dot{x}(t) + \proj_{C(x(t))}{(0)} = 0,
\end{align}
with $m, \gamma >0$, which was discussed in \cite{Attouch2015}. In the single objective setting ($m=1$), both \eqref{eq:IMOG'} and \eqref{eq:IMOG2} reduce  to the \emph{heavy ball with friction system}
\begin{align}
\label{eq:HBF}
\tag{HBF}
    m\ddot{x}(t) + \gamma\dot{x}(t) + \nabla f(x(t)) = 0,
\end{align}
which is well studied for different types of objective functions $f$, see e.g. \cite{Polyak1964, Alvarez2000, Attouch2000}.

We discretize \eqref{eq:IMOG'} to obtain an iterative scheme of the form
\begin{align*}
    x^{k+1} = x^k + a(x^k-x^{k-1}) - s\sum_{i=1}^m \theta_i^k \nabla f_i(x^k),
\end{align*}
with appropriately chosen coefficients $a, s > 0$ and $\theta^k \in \R^m$. This scheme can be interpreted as an inertial gradient method for \eqref{eq:MOP1}. We show that it shares many properties with its continuous counterpart and that iterates defined by this algorithm converge weakly to Pareto critical points. To the best of our knowledge this is the first multiobjective method involving a constant momentum term with guaranteed convergence to Pareto critical solutions.

In a further step, we introduce time-dependent friction and informally define the following \emph{multiobjective gradient-like system with asymptotically vanishing damping}
\begin{align}
\ddot{x}(t) + \frac{\alpha}{t}\dot{x}(t) + \proj_{C(x(t))}{(-\ddot{x}(t))} = 0.
\tag{MAVD}
    \label{eq:MAVD}
\end{align}
In the single objective setting this system simplifies to the following inertial system with \emph{asymptotically vanishing damping}
\begin{align}
    \ddot{x}(t) + \frac{\alpha}{t}\dot{x}(t) + \nabla f(x(t)) = 0.
    \tag{AVD}
    \label{eq:AVD}
\end{align}
It is well-known that \eqref{eq:AVD} is naturally linked with Nesterov's accelerted gradient method \cite{Su2014, Attouch2015_2, Attouch2018, Attouch2019}. Discretizing the dynamical system \eqref{eq:MAVD}, and using our knowledge about \eqref{eq:IMOG'}, we derive an accelerated gradient method for multiobjective optimization that takes the form
\begin{align*}
    x^{k+1} = x^k + \frac{k-1}{k+2}(x^k-x^{k-1}) - s\sum_{i=1}^m \theta_i^k \nabla f_i(x^k),
\end{align*}
with appropriately chosen coefficients $s > 0$ and $\theta^k \in \R^m$. In a recent preprint, Tanabe, Fukuda and Yamashita derive an accelerated proximal gradient method for multiobjective optimization using the concept of merit functions \cite{Tanabe2022_2}. We show that the method we derive from the differential equation \eqref{eq:MAVD} achieves the same convergence rate of order $\mathcal{O}(k^{-2})$ for the function values, measured with a so-called merit function.

This remainder of the paper is organized according to the novelties described above. After introducing some basic definitions and notation in Section \ref{sec:background}, we prove that solutions to the system \eqref{eq:IMOG'} exist in finite-dimensional Hilbert spaces in Section \ref{sec:global_existence}, and that they converge to Pareto critical points in Section \ref{sec:asymptotic analysis}. Based on that, we derive a discrete optimization algorithm from an explicit discretization on \eqref{eq:IMOG'} and show that the iterates defined by this method converge weakly to Pareto critical points, in Section \ref{sec:interntial_multi_grad_algo}. We then introduce Nesterov Acceleration and prove an improved convergence result in Section \ref{sec:acc_multi_grad_algo}, and a further improvement in numerical efficiency is discussed in Section \ref{sec:improving numerical efficiency}. These two central algorithms are summarized in Algorithms \ref{algo:ACC_GRAD} and \ref{algo:ACC_GRAD_wo_Q} in the respective sections. We compare these on convex and non-convex example problems in Section \ref{sec:numerical_experiments}, and conclude our findings and list future research projects in Section \ref{sec:conclusion}.
\section{Background}
\label{sec:background}
\subsection{Notations}
Throughout this paper, $\H$ is a real Hilbert space with inner product $\langle \cdot, \cdot \rangle$ and induced norm $\lVert \cdot \rVert$. The set $\Delta_m \coloneqq \left\lbrace \alpha \in \R^m \,\,:\,\, \alpha \ge 0, \,\,\text{and}\,\, \sum_{i=1}^m \alpha_i = 1 \right\rbrace$ is the unit simplex. We denote the open ball with radius $\delta > 0$ and center $x$ by $B_{\delta}(x) \coloneqq \left\lbrace y \in \H \,\, : \,\, \lVert y - x \rVert < \delta \right\rbrace$. The closed ball with radius $\delta > 0$ and center $x$ is denoted by $\overline{B_{\delta}(x)}$. For a set of vectors $\left\lbrace \xi_1, \cdots, \xi_k \right\rbrace \subseteq \H$ we denote the convex hull of these vectors by $\co(\left\lbrace \xi_1, \cdots, \xi_k \right\rbrace) \coloneqq \left\lbrace \sum_{i = 1}^m \alpha_i \xi_i \,\, : \,\, \alpha \in \Delta_m \right\rbrace$. For a closed convex set $C \subseteq \H$ the projection of a vector $x \in \H$ onto $C$ is $\proj_C(x) \coloneqq \argmin_{y \in \H} \lVert y - x \rVert^2$. For two vectors $x, y \in \R^k$, we define the partial order $x \le y :\Leftrightarrow x_i \le y_i$ for all $i = 1,\dots, k$. We define $\ge$, $<$, $>$ on $\R^k$ analogously. When we treat dynamical systems, $t \in \R$ and $x \in \H$ are the time and state variable, respectively. We denote trajectories in $\H$ with $t\mapsto x(t)$ with first derivative $\dot{x}(t)$ and second derivative $\ddot{x}(t)$.
\subsection{Multiobjective Optimization}
Consider the unconstrained multiobjective optimization problem
\leqnomode
\begin{align}
\tag{MOP}
    \min_{x \in \H} \left[ 
    \begin{array}{c}
        f_1(x)  \\
        \vdots \\
        f_m(x)
    \end{array}
    \right],
    \label{eq:MOP}
\end{align}
with at least once differentiable objective functions $f_i:\H \to \R$ for $i = 1, \dots, m$. The definitions in this subsection are aligned with \cite{Miettinen1998}.
\begin{mydef}[\cite{Miettinen1998},~pp.10-20]
Consider the optimization problem \eqref{eq:MOP}.
\begin{enumerate}[i)]
    \item A point $x^* \in \H$ is Pareto optimal if there does not exist another point $x \in \H$ such that $f_i(x) \le f_i(x^*)$ for all $i = 1,\dots,m,$ and $f_j(x) < f_j(x^*)$ for at least one index $j$. The set of all Pareto optimal points is the Pareto set, which we denote with $P$. 
    \item A point $x^* \in \H$ is locally Pareto optimal if there exists $\delta > 0$ such that $x^*$ is Pareto optimal in $B_{\delta}(x^*)$. 
    \item A point $x^* \in \H$ is weakly Pareto optimal if there does not exist another vector $x \in \H$ such that $f_i(x) < f_i(x^*)$ for all $i = 1,\dots, m$.
    \item A point $x^* \in \H$ is locally weakly Pareto optimal if there exists $\delta > 0$ such that $x^*$ is weakly Pareto optimal in $B_{\delta}(x^*)$.
\end{enumerate}
\end{mydef}
In this paper, we treat convex MOPs. Therefore, the objective functions $f_i$ are convex for all $i = 1, \dots, m$. In this setting, every locally (weakly) Pareto optimal point is also (weakly) Pareto optimal. For unconstrained MOPs, the so-called Karush-Kuhn-Tucker conditions can be written as follows.
\begin{mydef}
A point $x^* \in \H$ satisfies the Karush-Kuhn-Tucker conditions if there exists $\alpha \in \Delta_m$ such that $\sum_{i=1}^m \alpha_i \nabla f_i(x^*) = 0$. If $x^*$ satisfies the Karush-Kuhn-Tucker conditions, we call it Pareto critical.
\end{mydef}
The condition $0 \in \co\left(\left\lbrace \nabla f_i(x^*) \,\,:\,\, i = 1, \dots, m \right\rbrace \right)$ is equivalent to the Karush-Kuhn-Tucker conditions. Analogously to the single objective setting, criticality of a point is a necessary condition for optimality. In the convex setting, the KKT conditions are also sufficient conditions for weak Pareto optimality. If we denote the Pareto set by $P$, the weak Pareto set by $P_w$ and the Pareto critical set by $P_c$ in the setting of smooth and convex multiobjective optimization, we observe the relation
\begin{align*}
    P \subset P_w = P_c.
\end{align*}
\subsection{Accelerated Methods for Multiobjective Optimization}
Accelerated methods for multiobjective optimization are not sufficiently discussed from a theoretical point of view in the literature yet. In \cite{ElMoudden2021} El Moudden and El Moutasim propose an accelerated method for multiobjective optimization which incorporates the multiobjective descent direction by Fliege \cite{Fliege2000} and the same acceleration scheme as in Nesterov's accelerated method \cite{Nesterov1983}. El Moudden and El Moutasim prove a convergence rate of the function values with rate $\mathcal{O}(k^{-2})$. Their proof relies on the restrictive assumption that the Lagrange multipliers of the quadratic subproblem, that is used to compute the step direction in every iteration, remain fixed from a certain point on. Under this assumption, the method simplifies to Nesterov's method for single objective optimization problems applied to a weighted sum of the objective functions with fixed weights. Only recently, Tanabe, Fukuda and Yamashita derived an accelerated proximal gradient method for multiobjective optimization problems in \cite{Tanabe2022_2}. They developed their method using the concept of merit functions (see Subsection \ref{subsec:merit_functions}) and show that the function values converge with rate $\mathcal{O}(k^{-2})$ without additional assumptions on the Lagrange multipliers.
\subsection{Dynamical Systems linked to Multiobjective Optimization}
\label{subsec:in_grad_sys_mutli_obj}
In \cite{Smale1973} Smale presents the idea of treating multiobjective optimization problem with a continuous time perspective that is motivated from an economical point of view using utility functions in a multi-agent framework. The simplest dynamical system for multiobjective optimization problems is the \emph{multiobjective gradient system}
\begin{align}
    \label{eq:MOG}
    \tag{MOG}
    \dot{x}(t) + \proj_{C(x(t))}(0) = 0,
\end{align}
where $C(x(t)) = \co \left( \lbrace \nabla f_i(x(t)) \,:\, i = 1, \dots, m \rbrace \right)$. This system is already treated in \cite{Henry1973} and in addition by Cornet in \cite{Cornet1979, Cornet1981, Cornet1983} as a dynamical system for resource allocation from an economical point of view. In \cite{Miglierina2004, Attouch2014} the system \eqref{eq:MOG} gets introduced as a tool for multiobjective optimization. The system \eqref{eq:MOG} can also be seen as a continuous version of the multiobjective steepest descent method by Fliege \cite{Fliege2000}. In the single objective setting ($m=1$), the system \eqref{eq:MOG} simplifies to the \emph{steepest descent dynamical system} $\dot{x}(t) + \nabla f(x(t)) = 0$. Generalizations of \eqref{eq:MOG} are treated in \cite{Attouch2014, Attouch2015_3}. In \cite{Attouch2014} Attouch and Goudou discuss a dynamical system for constrained minimization and in \cite{Attouch2015_3} Attouch, Garrigos and Goudou present a differential inclusion for constrained nonsmooth optimization.\\
In \cite{Attouch2015}, Attouch and Garrigos introduce inertia in the system \eqref{eq:MOG} and define the following \emph{inertial multiobjective gradient-like dynamical system}
\begin{align}
\label{eq:IMOG}
\tag{IMOG}
    m\ddot{x}(t) + \gamma \dot{x}(t) + \proj_{C(x(t))}(0) = 0.
\end{align}
Trajectories of \eqref{eq:IMOG} converge weakly to Pareto optimal solutions given $\gamma^2 > mL$, where $L$ is a common Lipschitz constant for the gradients of the objective functions.
\subsection{Merit Functions}
\label{subsec:merit_functions}
A merit function associated with an optimization problem is a function that returns zero at an optimal solution and which is strictly positive otherwise. An overview on merit functions used in multiobjective optimization is given in \cite{Tanabe2020}. In our proofs we use the merit function
\reqnomode
\begin{align}
    u_0(x) \coloneqq \sup_{z \in \H} \min_{i = 1,\dots, m} f_i(x) - f_i(z),
    \label{eq:merit_function}
\end{align}
which satisfies the following statement.
\begin{theorem}[\cite{Tanabe2020} Theorem 3.1]
It holds that $u_0(x)\ge 0$ for all $x \in \H$. Moreover, $x \in \H$ is weakly Pareto optimal for \eqref{eq:MOP}, if and only if $u_0(x) = 0$.
\end{theorem}
Additionally, $u_0(x)$ is lower semicontinuous. Therefore, if $(x^k)_{k \ge 0}$ is a sequence with $u_0(x^k) \to 0$, every cluster point of $(x^k)_{k \ge 0}$ is weakly Pareto optimal. This motivates the usage of $u_0(x)$ as a measure of complexity for multiobjective optimization methods. The function $u_0(x)$ is not the only merit function for multiobjective optimization problems, see also \cite{Chen2000, Yang2002, Liu2009} and further references in \cite{Tanabe2020}.
\section{Global Existence in Finite Dimensions}
\label{sec:global_existence}
In this section, we show that solutions exist for the Cauchy problem related to
\eqref{eq:IMOG'}, i.e
\leqnomode
\begin{align}
\label{eq:CP}
\tag{CP}
    \left\lvert 
    \begin{array}{l}
        \ddot{x}(t) + \alpha\dot{x}(t) + \proj_{C(x(t))}{(-\ddot{x}(t))} = 0,  \\
        \\
        x(0) = x_0, \quad \dot{x}(0) = v_0,
    \end{array}
    \right.
\end{align}
\reqnomode
with initial data $x_0, v_0 \in \H$.
To this end, we show that for this system solutions exist if there exists a solution for a first order differential inclusion
\begin{align*}
    (\dot{u}, \dot{v}) \in F(u,v),
\end{align*}
with a set-valued mapping $F:\H\times\H \rightrightarrows \H\times\H$. Then, we use an existence theorem for differential inclusions from \cite{Aubin2012}. Our argument works only in finite-dimensional Hilbert spaces. Thus, we assume $\dim(\H) < +\infty$ from here on. In our context, the following set-valued mapping is of interest:
\begin{align}
\label{eq:set_valued_F}
    F: \H \times \H \rightrightarrows \H\times \H, \quad (u,v) \mapsto \lbrace v \rbrace \times \left(\big( - \argmin_{z\in C(u)} \langle z, v \rangle \big) - \alpha v \right)
\end{align}
As stated above, $C(u) \coloneqq \co \left(\lbrace \nabla f_i(u) \, : \, i = 1,\dots,m \rbrace \right)$. We can show that \eqref{eq:CP} has a solution if the differential inclusion
\leqnomode
\begin{align}
    \label{eq:DI}
    \tag{DI}
    \left\lvert 
    \begin{array}{l}
        (\dot{u}(t), \dot{v}(t)) \in F(u,v),  \\
        \\
        (u(0), v(0)) = (u_0, v_0),
    \end{array}
    \right.
\end{align}
\leqnomode
with appropriate initial data $u_0, v_0$ has a solution.
\subsection{Existence of Solutions to (DI)}
\label{subsec:ex_sol_DI}
To show that there exist solutions to \eqref{eq:DI}, we investigate the set-valued mapping $(u,v) \rightrightarrows F(u,v)$ defined in \eqref{eq:set_valued_F}. The basic definitions for set-valued mappings used in this subsection can be found in \cite{Aubin2012}. 
\begin{myprop}
For all $(u,v) \in \H \times \H$, $F(u,v) \subset \H\times\H$ is convex and compact. 
\label{prop:set_valued_F_compact_image}
\end{myprop}
\begin{proof}
The statement follows directly from the definition.
\end{proof}
To use an existence theorem from \cite{Aubin2012}, we need to show that $(u,v) \rightrightarrows F(u,v)$ is upper semicontinuous. Showing this is elementary. We omit the full proof here but sketch a possible way to prove this result.
\begin{mylemma}
Let $C(u) \coloneqq \co(\lbrace c_i(u)\, :\, i = 1,\dots, m\rbrace)$ with $c_i:\H \to\H,\, u\to c_i(u)$ continuous for $i = 1,\dots, m$. Let $(u_0, v_0) \in \H \times \H$ be fixed.\\
Then, for all $\varepsilon > 0$ there exists an $\delta > 0$ such that for all $(u,v) \in \H\times\H$ with $\lVert u - u_0 \rVert < \delta$ and $\lVert v - v_0 \rVert < \delta$ and for all $z \in \argmin_{z \in C(u)} \langle z, v \rangle$ there exists $z_0 \in \argmin_{z_0 \in C(u_0)} \langle z_0, v_0 \rangle$ with $\lVert z - z_0 \rVert < \varepsilon$.
\label{lem:lin_opt_pert}
\end{mylemma}
\begin{proof}
The proof follows straightforward.
\end{proof}
\begin{myprop}
The set-valued mapping $(u,v) \rightrightarrows F(u,v)$ is upper semicontinuous.
\label{prop:F_usc}
\end{myprop}
\begin{proof}
Using Lemma \ref{lem:lin_opt_pert} we can show in a straightforward manner
\begin{align*}
    F((u_0, v_0) + B_{\delta}(0)) \subset F(u_0, v_0) + B_{\varepsilon}(0)
\end{align*}
using only continuity arguments. Then, the statement follows by the fact that $(u,v) \rightrightarrows F(u,v)$ is locally compact.
\end{proof}
\begin{myprop}
Let $\H$ have finite dimension. Then, the mapping
\begin{align*}
\phi : \H\times \H \to \H \times \H,\quad (u,v) \mapsto \left(v,\, \proj_{F(u,v)}(0) \right),
\end{align*}
is locally compact. 
\label{prop:loc_comp}
\end{myprop}
\begin{proof}
If $\dim(\H) < + \infty$ the proof follows easily since all images $F(u,v)$ are compact and depend on $(u,v)$ in a well behaved manner. On the other hand, from $\phi$ being locally compact, we get that $v \mapsto v$ is locally compact which is equivalent to $\H$ being finite-dimensional. 
\end{proof}
The following existence theorem from \cite{Aubin2012} is applicable in our setting.
\begin{theorem}[\cite{Aubin2012}, p.~98, Theorem 3]
Let $X$ be a Hilbert space, $\Omega \subset \R \times X$ be an open subset containing $(0,x_0)$. Let $F$ be an upper semicontinuous map from $\Omega$ into the non-empty closed convex subsets of $X$. We assume that $(t,x) \mapsto \proj_{F(t,x)}(0)$ is locally compact. Then, there exists $T > 0$ and an absolutely continuous function $x(\cdot)$ defined on $[0, T]$ which is a solution to the differential inclusion
\begin{align*}
    \dot{x}(t) \in F(t, x(t)), \quad x(0) = x_0.
\end{align*}
\label{thm:existence}
\end{theorem}
We are finally in the position to state an existence theorem for \eqref{eq:DI}. 
\begin{theorem}
Assume $\H$ is finite-dimensional and that the gradients of the objective function $\nabla f_i$ are globally Lipschitz continuous. Then, for all $(u_0, v_0) \in \H \times \H$ there exists $T > 0$ and an absolutely continuous function $(u(\cdot), v(\cdot))$ defined on $[0, T]$ which is a solution to the differential inclusion \eqref{eq:DI}.
\label{thm:DI_sol_exist_fin}
\end{theorem}
\begin{proof}
The proof follows immediately from Propositions \ref{prop:F_usc} and \ref{prop:loc_comp} which show that the set-valued mapping satisfies all conditions required for Theorem \ref{thm:existence}. 
\end{proof}
In the following, we show that under additional conditions on the objective functions $f_i$ there exist solutions defined on $[0, +\infty)$. The extension of the solution works using a standard argument. We show that the solutions to \eqref{eq:DI} remain bounded. Then, we use Zorn's Lemma to retrieve a contradiction if there is a maximal solution that is not defined on $[0, +\infty)$.
\begin{theorem}
Assume $\H$ is finite-dimensional and that the gradients of the objective function $\nabla f_i$ are globally Lipschitz continuous. Then, for all $(u_0, v_0) \in \H \times \H$ there exists an absolutely continuous function $(u(\cdot), v(\cdot))$ defined on $[0, +\infty)$ which is a solution to the differential inclusion \eqref{eq:DI}.
\label{thm:DI_sol_exist}
\end{theorem}
\begin{proof}
Theorem \ref{thm:DI_sol_exist_fin} guarantees the existence of solutions defined on $[0, T)$ for some $T \ge 0$. Using the domain of definition, we can define a partial order on the set of solutions to the problem \eqref{eq:DI}. Assuming there is no solution defined on $[0, +\infty)$, Zorn's Lemma guarantees the existence of a solution $(u(\cdot), v(\cdot)):[0, T) \to \H \times \H$ with $T < + \infty$ which can not be extended. We will show that $(u(t), v(t))$ does not blow up in finite time and therefore can be extended which contradicts the claimed maximality.\\
Define
\begin{align*}
    h(t) \coloneqq \lVert (u(t), v(t)) - (u(0), v(0)) \rVert_{\H \times \H},
\end{align*}
where $\lVert (x,y) \rVert_{\H \times \H} = \sqrt{\lVert x \rVert^2 + \lVert y \rVert^2}$. We show that $h(t)$ can be bounded by a real-valued function. Using the Cauchy-Schwarz inequality, we directly get
\reqnomode
\begin{align}
\label{eq:derivative_h_sq}
    \frac{d}{dt} \frac{1}{2} h^2(t) &= \langle (\dot{u}(t), \dot{v}(t)), (u(t), v(t)) - (u(0), v(0)) \rangle\\
    &\le \lVert F(u(t), v(t))\rVert h(t).
\end{align}
We next derive a bound on $\lVert F(u, v) \rVert$. The basic inequalities between the $\ell_1$ and $\ell_2$ norm applied to $\left( \lVert x \rVert, \lVert y \rVert \right) \in \R^2$ yield
\begin{align*}
    \lVert (x,y) \rVert_{\H \times \H} \le \lVert x \rVert + \lVert y \rVert \le \sqrt{2} \lVert (x,y) \rVert_{\H \times \H}.
\end{align*}
 Therefore
\begin{align*}
    &\lVert F(u, v) \rVert_{\H \times \H} \le \lVert v \rVert + \lVert \big(- \argmin_{z\in C(u)} \langle z, v \rangle \big) - \alpha v\rVert\\
    \le &(1 + \alpha) \lVert v \rVert + \max_{\theta \in \Delta_m} \lVert \sum_{i=1}^m \theta_i \nabla f_i(u) \rVert  \\
    \le & (1+ \alpha) \lVert v \rVert + \max_{\theta \in \Delta_m} \lVert \sum_{i=1}^m \theta_i \left(\nabla f_i(u) - \nabla f_i(0)\right) \rVert + \max_{\theta \in \Delta_m} \lVert \sum_{i=1}^m \theta_i \nabla f_i(0) \rVert \\
    \le & (1+ \alpha) \lVert v \rVert + L \lVert u \rVert + \max_{i=1,\dots,m} \lVert \nabla f_i(0) \rVert \\
    \le & c(1 + \lVert (u,v) \rVert_{\H \times \H}),
\end{align*}
where we have chosen $c = \sqrt{2}\max\left\lbrace 1 + \alpha, L, \max_{i=1,\dots,m} \lVert \nabla f_i(0) \rVert\right\rbrace$.
Combining the last inequality with \eqref{eq:derivative_h_sq}, we get
\begin{align*}
    \frac{d}{dt}\frac{1}{2}h^2(t) \le c (1 + h(t))h(t),\,\, \text{for all} \,\, t \in [0, T).
\end{align*}
Using a Gronwall-type argument (see Lemma A.4 and Lemma A.5 in \cite{Brezis1973}) just as in Theorem 3.5 in \cite{Attouch2015}, we know that there exists $c\in \R$ such that for an arbitrary $\varepsilon > 0$
\begin{align*}
    h(t) \le cT \exp(cT),\,\, \text{  for all  }\,\, t \in [0, T - \varepsilon].
\end{align*}
Since this upper bound is independent of $t$ and $\varepsilon$, it follows that $h \in L^{\infty}([0,T], \R)$. Therefore, solutions to \eqref{eq:DI} do not blow up in finite time and can be extended. This is a contradiction to the maximality of the solution $(u(t), v(t))$.
\end{proof}
\subsection{Existence of Solutions to (CP)}
\label{subsec:ex_sol_CP}
Using the findings of the previous subsection, we can proceed with the discussion of the Cauchy problem \eqref{eq:CP}. In this subsection, we show that solutions to the differential inclusion \eqref{eq:DI} immediately give solutions to the Cauchy problem \eqref{eq:CP}. 
\begin{theorem}
Let $x_0, v_0 \in \H$. Assume that $(u(t), v(t))$ for $t\in [0, +\infty)$ is a solution to \eqref{eq:DI} with $(u(0), v(0)) = (x_0, v_0)$. Then, it follows that $x(t) \coloneqq u(t)$ satisfies the differential equation
\begin{align*}
    \ddot{x}(t) + \alpha \dot{x}(t) + \proj_{C(x(t))}(-\ddot{x}(t))= 0,\,\, \text{for all}\,\, t\in (0, +\infty),
\end{align*}
and $x(0) = x_0$, $\dot{x}(0) = v_0$.
\label{thm:DI_sol_is_CP_sol}
\end{theorem}
\begin{proof}
Since $(u(t), v(t))$ is a solution to \eqref{eq:DI}, it holds for all $t \in (0, +\infty)$ that
\begin{align*}
    \dot{u}(t) &= v(t), \\
    \dot{v}(t) &\in \argmin_{z \in C(u(t))} \langle z, v(t)\rangle - \alpha v(t).
\end{align*}
Using  Lemma \ref{lem:proj_1}, the second line gives $- \alpha v(t) = \proj_{C(u(t)) + \dot{v}(t)}(0)$, which is equivalent to
\begin{align*}
    \dot{v}(t) + \alpha v(t) + \proj_{C(u(t))}(-\dot{v}(t)) = 0.
\end{align*}
Rewriting this system using $x(t) = u(t)$, $\dot{x}(t) = \dot{u}(t) = v(t)$ and $\ddot{x} = \dot{v}(t)$ and verifying the initial conditions $x(0) = u(0) = x_0$ and $\dot{x}(0) = v(0) = v_0$ yields the desired result.
\end{proof}
Finally, we can state the full existence theorem for the Cauchy problem \eqref{eq:CP}. 
\begin{theorem}
Assume $\H$ is finite-dimensional and that the gradients of the objective function $\nabla f_i$ are globally Lipschitz continuous. Then, for all $x_0, v_0 \in \H$, there exists a twice continuously differentiable function $x(t)$ defined on $[0, +\infty)$ which is absolutely continuous with absolutely continuous first derivative $\dot{x}(t)$, and which is a solution to the Cauchy problem \eqref{eq:CP} with initial values $(x_0, v_0)$.
\end{theorem}
\begin{proof}
The proof follows immediately combining Theorem \ref{thm:DI_sol_exist} and Theorem \ref{thm:DI_sol_is_CP_sol}.
\end{proof}
\begin{myremark}
Throughout this section, we have assumed that the gradients $\nabla f_i$ of the objective functions are globally Lipschitz continuous. One can relax this condition and only require the gradients to be Lipschitz continuous on bounded sets if we can guarantee that the solutions remain bounded. This holds for example if one of the objective functions $f_i$ is coercive.
\end{myremark}
\section{Asymptotic Analysis of Trajectories of (IMOG')}
\label{sec:asymptotic analysis}
In this section, we omit the assumption $\dim(\H) < + \infty$. We show that trajectories of the differential equation \eqref{eq:IMOG'} converge weakly to Pareto critical points of the optimization problem \eqref{eq:MOP}. This follows from a dissipative property of the system and an argument that relies on Opial's Lemma. We first define a so-called energy function for the system \eqref{eq:IMOG'} that has Lyapunov-type properties.
\begin{myprop}
    Let $x:[0, +\infty) \to \H$ be a solution to \eqref{eq:CP}. For $i = 1,\dots,m$ define the global energy
    \begin{align*}
        \E_i:[0, T) \to \R, \quad t\mapsto f_i(x(t)) + \frac{1}{2}\lVert \dot{x}(t) \rVert^2.
    \end{align*}
    Then, for all $t\in(0,+\infty)$ it holds that $\frac{d}{dt}\E_i(t) \le -\alpha\lVert \dot{x}(t) \rVert^2$.
\label{prop:dissipative_const}
\end{myprop}
\begin{proof}
From the definition of the differential equation \eqref{eq:IMOG'}, it follows that $-\alpha\dot{x}(t) = \proj_{\co\left(\nabla f_i(x(t)) + \ddot{x}(t)\right)}(0)$. By the variational characterization of the convex projection, we get for all $i = 1,\dots, m$
\begin{align*}
    \langle \alpha\dot{x}(t) + \nabla f_i(x(t)) + \ddot{x}(t) , \alpha\dot{x}(t) \rangle \le 0,
\end{align*}
which immediately gives
\begin{align*}
    \langle \nabla f_i(x(t)), \dot{x}(t) \rangle + \langle \dot{x}(t), \ddot{x}(t) \rangle \le -\alpha \lVert \dot{x}(t) \rVert^2.
\end{align*}
Applying the chain rule to $\frac{d}{dt}\E_i(t)$ yields the desired result.
\end{proof}
\begin{myprop}
Let $x:[0, +\infty) \to H$ be a bounded solution of \eqref{eq:CP} and let further $\nabla f_i$ be Lipschitz continuous on bounded sets. Then, for all $i = 1, \dots, m$
\begin{enumerate}[i)]
    \item $\lim_{t\to +\infty}\E_i(t) = \E_i^{\infty} > -\infty$
    \item $\dot{x} \in L^2([0, +\infty)) \cap L^{\infty}([0, +\infty))$
    \item $\ddot{x} \in L^{\infty}([0, +\infty))$, $\lim_{t \to + \infty} \lVert \dot{x}(t) \rVert = 0$ and $\lim_{t \to + \infty} f_i(x(t)) = \E_i^{\infty}$
    \item There exists a monotonically increasing unbounded sequence $(t_k)_{k\ge0}$ with\\  $\proj_{C(x(t_k))}(0) \to 0$ for $k \to +\infty$.
\end{enumerate}
\label{prop:energy_estimation_const}
\end{myprop}
\reqnomode
\begin{proof}
i)\quad From Proposition \ref{prop:dissipative_const}, we immediately get that $\E_i$ is monotonically decreasing and therefore $\E_i(t) \to \E_i^{\infty}$ as $t \to +\infty$. We have to show that in fact $\E_i^{\infty} > - \infty$. Since $\nabla f_i$ is bounded on bounded sets, we can conclude by the mean value theorem that $f_i$ is bounded on bounded sets. Since $x(t)$ remains bounded by assumption, we conclude that $f_i(x(t))$ is bounded from below, and hence
\begin{align*}
    \E_i^{\infty} \ge \inf_{t\ge 0} f_i(x(t)) > - \infty.
\end{align*}
ii)\quad We know that $f_i(x(t))$ is bounded. Then, by the definition of $\E_i$ and the fact that $\E_i$ is monotonically decreasing, we immediately get that $\dot{x}$ is bounded for all $t\ge 0$. Since it is of class $C^1$, it follows that $\dot{x} \in L^{\infty}([0, +\infty))$. Using Proposition \ref{prop:dissipative_const} it follows that
\begin{align*}
    \alpha\int_{0}^{+\infty} \lVert \dot{x}(t)\rVert^2\, dt\le - \int_{0}^{+\infty} \frac{d}{dt}\E_i(s)\,ds = \E_i(0) - \E_i^{\infty} < +\infty,
\end{align*}
and therefore $\dot{x} \in L^2([0, +\infty))$.\\ \\
iii)\quad Since $\dot{x}(t)$ and $\nabla f_i(x(t))$ remain bounded for all $t \ge 0$ it follows that $\ddot{x}(t) = -\alpha \dot{x}(t) - \proj_{C(x(t))}(-\ddot{x}(t))$ remains bounded for all $t\ge 0$. By the fact that $x$ is twice continuously differentiable, it follows that $\ddot{x}$ is continuous and hence $\ddot{x} \in L^{\infty}([0, + \infty))$. Then, from $\dot{x} \in C^1([0, +\infty)) \cap L^2([0, +\infty))$ and $\ddot{x} \in C([0, +\infty)) \cap L^{\infty}([0, + \infty))$ it follows that $\lim_{t \to + \infty} \lVert \dot{x}(t) \rVert = 0$. From $\lim_{t \to + \infty} \lVert \dot{x}(t) \rVert = 0$ and part i) we can immediately conclude $\lim_{t \to + \infty} f_i(x(t)) = \E_i^{\infty}$.\\ \\
iv) \quad Assume that the negation of statement iv) holds, namely
\begin{align}
    \exists M > 0 \,\, \exists T \ge 0 \,\, \forall t \ge T \,\, : \,\, \lVert \proj_{C(x(t))}(0) \rVert \ge M.
    \label{eq:condition_M}
\end{align}
Fix an arbitrary $\delta > 0$. Since $\dot{x}(t) \to 0$ and $\nabla f_i$ is Lipschitz continuous on a set containing $x(t)$ it follows that there exists $T_{\delta} > T$ such that for all $t > T_{\delta}$ and $s \in [t, t+\delta)$
\begin{align}
\begin{split}
    \lVert \nabla f_i(x(s)) - \nabla f_i(x(t)) \rVert < \frac{M}{8} \,\,\text {and}\,\, \lVert \alpha \dot{x}(s) \rVert < \frac{M}{8}.
\end{split}
\label{eq:condition_M_8}
\end{align}
Define $v \coloneqq \proj_{C(x(t))}/ \lVert \proj_{C(x(t))} \rVert$. From \eqref{eq:condition_M} it follows that
\begin{align*}
    \forall \xi \in C(x(t)) \,\, : \,\, \langle \xi , v \rangle \ge M.
\end{align*}
Combining the last statement with \eqref{eq:condition_M_8} and using the Cauchy-Schwarz inequality we get
\begin{align*}
    \forall s \in [t, t+\delta) \,\,\forall \xi \in C(x(s)) \,\, : \,\, \langle \xi + \alpha \dot{x}(s)  , v \rangle \ge \frac{M}{4}.
\end{align*}
And hence
\begin{align*}
    \forall s \in [t, t+\delta) \,\, : \,\, \langle - \ddot{x}(s)  , v \rangle \ge \frac{M}{4}.
\end{align*}
Using the Cauchy-Schwarz inequality again, we get
\begin{align*}
    & \lVert \dot{x}(t) - \dot{x}(t+\delta) \rVert \ge \langle \dot{x}(t) - \dot{x}(t + \delta) , v \rangle = \langle \int_{t}^{t+\delta} - \ddot{x}(s) \,ds, v \rangle \\
    = & \int_{t}^{t+\delta} \langle - \ddot{x}(s), v \rangle \,ds \ge \int_{t}^{t+\delta} \frac{M}{4} \,ds = \frac{M \delta}{4}.
\end{align*}
Since we can choose $\delta$ arbitrarily large and independently from $M$, this contradicts $\dot{x}(t) \to 0$.
\end{proof}
We will use part iv) of Proposition \ref{prop:energy_estimation_const} to show that a weak limit point of the trajectory $x(t)$ is Pareto critical. To this end we introduce the following lemma that states a demiclosedness property of the set-valued mapping $C: \H \rightrightarrows \H,\, x \mapsto C(x)$.
\begin{mylemma}[\cite{Attouch2015_2} Lemma 2.4, \cite{Attouch2015} Lemma 4.10]
Assume that the objective functions $f_i$ are continuously differentiable. Let $(x^k)_{k \ge 0}$ be a sequence in $\H$ that converges weakly to $x^{\infty}$, and assume there exists a sequence $(g^k)_{k \ge 0}$ with $g^k \in C(x^k)$ that converges strongly to zero. Then, $0 \in C(x^{\infty})$ and hence $x^{\infty}$ is Pareto critical.
\label{lem:demiclosedness_co_grad_f}
\end{mylemma}
If we can show that the trajectories of \eqref{eq:IMOG'} converge weakly, Proposition \ref{prop:energy_estimation_const} together with Lemma \ref{lem:demiclosedness_co_grad_f} guarantees that the limit points are Pareto critical. To show that the trajectories are in fact converging, we require Opial's Lemma.
\begin{mylemma}[Opial's Lemma \cite{Opial1967}]
Let $S \subset \H$ be a nonempty subset of $\H$ and \\$x:[0, +\infty) \to \H$. Assume that $x(t)$ satisfies the following conditions.
\begin{enumerate}[i)]
    \item Every weak sequential cluster point of $x(t)$ belongs to $S$.
    \item For every $z\in S$, $\lim_{t\to + \infty} \lVert x(t) - z \rVert$ exists.
\end{enumerate}
Then, $x(t)$ converges weakly to an element $x^{\infty} \in S$.
\end{mylemma}
To use Opial's Lemma, we need a suitable nonempty set $S \subset \H$ that we define in the following proposition.
\begin{myprop}
Let $x(t)$ be a bounded solution to \eqref{eq:CP}. Then, the set
\begin{align}
\label{eq:set_s_for_opial}
    S \coloneqq \left\lbrace z \in \H \,:\, f_i(z) \le \E_i^{\infty} \,\,\text{for all  } i = 1,\dots, m,\right\rbrace,
\end{align}
is nonempty.
\label{prop:S_nonempty}
\end{myprop}
\begin{proof}
Part iii) of Proposition \ref{prop:energy_estimation_const} states that $\lim_{t\to +\infty} f_i(x(t)) = \E_i^{\infty}$ for all $i = 1,\dots, m$. Since $x(t)$ is bounded, it possesses at least one weak sequential cluster point $x^{\infty}$. The objective functions $f_i$ are convex and continuous and therefore weakly lower semicontinuous. From this we conclude $x^{\infty} \in S$.
\end{proof}
For the set $S$ defined in \eqref{eq:set_s_for_opial} and a bounded solution $x(t)$ of \eqref{eq:CP}, the first part of Opial's Lemma is easy to obtain. It follows analogously to the proof of Proposition \ref{prop:S_nonempty} where it is shown that $S$ is nonempty. To show the second part of Opial's Lemma, we verify that $h_z(t) \coloneqq \frac{1}{2}\lVert x(t) - z \rVert^2$ satisfies a differential inequality. Then, the convergence can be deduced from the following lemma.
\begin{mylemma}[\cite{Attouch2000} Lemma 4.2]
Let $h \in C^1([0, +\infty), \R)$ be a positive function satisfying $\alpha \dot{h}(t) + \ddot{h}(t) \le g(t)$ for all $t \ge 0$, with $g \in L^1([0, +\infty), \R)$ and $\alpha > 0$. Then, $\lim_{t\to +\infty} h(t)$ exists.
\label{lem:diff_ineq}
\end{mylemma}
With these ingredients, we can formulate the main convergence theorem of this subsection.
\begin{theorem}
Assume that the objective functions $f_i$ are convex with gradients $\nabla f_i$ that are Lipschitz continuous on bounded sets. Then, solution $x:[0, +\infty) \to \H$ of \eqref{eq:CP} with arbitrary initial conditions $x^0, v^0 \in \H$ that remains bounded converges weakly to a Pareto critical point of \eqref{eq:MOP}.
\label{thm:x_t_weak_conv_pareto_critical}
\end{theorem}
\begin{proof}
For $z\in S$ define
\begin{align*}
    h_z(t) \coloneqq \frac{1}{2}\lVert x(t) - z\rVert ^2.
\end{align*}
Using the chain rule, we compute the first and the second derivative of $h_z(t)$ as
\begin{align*}
    \dot{h}_z(t) = \langle x(t) - z, \dot{x}(t)\rangle, \quad
    \ddot{h}_z(t) = \langle x(t) - z, \ddot{x}(t)\rangle + \lVert \dot{x}(t)\rVert^2.
\end{align*}
For a fixed $t \in (0,+\infty)$, write
\begin{align*}
    \alpha\dot{h}_z(t) + \ddot{h}_z(t) = \langle \ddot{x}(t) + \alpha \dot{x}(t), x(t) - z \rangle + \lVert \dot{x}(t) \rVert^2.
\end{align*}
Using the definition of \eqref{eq:IMOG'}, we can write $\ddot{x}(t) + \alpha \dot{x}(t) = - \sum_{i = 1}^m \theta_i \nabla f_i(x(t))$ for some weights $\theta \in \Delta_m$. Then, we write
\reqnomode
\begin{align}
\label{eq:weak_conv_eq_1}
    \alpha\dot{h}_z(t) + \ddot{h}_z(t) = \sum_{i=1}^m \theta_i \langle \nabla f_i(x(t)), z - x(t) \rangle + \lVert \dot{x}(t) \rVert^2.
\end{align}
Proposition \ref{prop:dissipative_const} gives for all $i = 1,\dots, m$
\begin{align*}
    \E_i(t) = f_i(x(t)) + \frac{1}{2} \lVert \dot{x}(t) \rVert^2 \ge \E_i^{\infty} \ge f_i(z) \ge f_i(x(t)) + \langle \nabla f_i(x(t)), z - x(t) \rangle,
\end{align*}
and therefore
\begin{align}
\label{eq:weak_conv_eq_2}
    \sum_{i = 1}^m \theta_i \langle \nabla f_i(x(t)), z - x(t) \rangle \le \frac{1}{2} \lVert \dot{x}(t) \rVert^2.
\end{align}
Combining inequalities \eqref{eq:weak_conv_eq_1} and \eqref{eq:weak_conv_eq_2} we get
\begin{align*}
    \alpha \dot{h}_z(t) + \ddot{h}_z(t) \le \frac{3}{2} \lVert \dot{x}(t) \rVert^2.
\end{align*}
By Proposition \ref{prop:energy_estimation_const}, we know $\lVert \dot{x}(\cdot) \rVert^2 \in L^1([0, +\infty))$. Then, Lemma \ref{lem:diff_ineq} states that $\lim_{t\to +\infty} h_z(t)$ exists. In addition, we know that every weak sequential cluster point of $x(t)$ belongs to $S$ by the weak lower semicontinuity of the objective functions $f_i$. Then, we can use Opial's Lemma \ref{lem:opial} to prove that $x(t)$ converges weakly to an element in $S$. Let $x^{\infty}$ be the weak limit of $x(t)$. Then, by Proposition \ref{prop:energy_estimation_const}, there exists a monotonically increasing unbounded sequence $(t_k)_{k \ge 0}$ with $\proj_{C(x(t_k))}(0) \to 0$ for $k \to +\infty$. Since $x(t_k)$ converges weakly to $x^{\infty}$, Lemma \ref{lem:demiclosedness_co_grad_f} states that $x^{\infty}$ is Pareto critical.
\end{proof}
\reqnomode
\section{An Inertial Multiobjective Gradient Algorithm}
\label{sec:interntial_multi_grad_algo}
In this section, we derive an inertial first order method for multiobjective optimization problems from an explicit discretization of the differential equation \eqref{eq:IMOG'}. We write the system \eqref{eq:IMOG'} in the equivalent form
\begin{align*}
    \alpha \dot{x}(t) + \proj_{C(x(t)) + \ddot{x}(t)}(0),
\end{align*}
and use the following discretization of the differential equation
\begin{align*}
\begin{split}
    \alpha\frac{x^{k+1} - x^{k}}{h} + \proj_{C(x^k) + \frac{x^{k+1} - 2x^k + x^{k-1}}{h^2}}(0) = 0,\\
    \alpha h (x^{k+1} - x^{k}) + \proj_{h^2C(x^k) + x^{k+1} - 2x^k + x^{k-1}} (0) = 0.
\end{split}
\label{eq:proj_rel_imog}
\end{align*}
Lemma \ref{lem:proj_2} states that $x^{k+1}$ is uniquely defined as
\begin{align}
\begin{split}
    x^{k+1} & = -\left( \frac{1}{1+\alpha h} \proj_{h^2 C(x^k) - 2x^k + x^{k-1}}(-x^k) - \frac{\alpha h}{1 + \alpha h} x^k\right) \\
    & = -\left( \frac{1}{1+\alpha h}\left[- x^k +  \proj_{h^2 C(x^k) - x^k + x^{k-1}}(0)\right] - \frac{\alpha h}{1 + \alpha h} x^k\right) \\
    & = x^k -  \frac{1}{1+\alpha h} \proj_{h^2 C(x^k) - x^k + x^{k-1}}(0).
\end{split}
\end{align}
Therefore, $x^{k+1}$ can be written as
\begin{align}
    x^{k+1} = x^k + \frac{1}{1 + \alpha h}(x^k - x^{k-1}) - \frac{h^2}{1+\alpha h} \sum_{i = 1}^m \theta_i^k \nabla f_i(x^k),
\label{eq:inertial_update}
\end{align}
where $\theta^k \in \Delta_m$ is the solution to the quadratic optimization problem
\begin{align}
\label{eq:QOP_intertial}
\begin{split}
    \min_{\theta \in \R^m} \, \left\lVert h^2 \left(\sum_{i=1}^m \theta_i \nabla f_i(x^k) \right) - (x^k - x^{k-1}) \right\rVert^2\,\,
    \st \, \theta \ge 0 \,\,\text{and}\,\,  \sum_{i = 1}^m \theta_i = 1.
\end{split}
\end{align}
In the following subsection, we analyze the asymptotic behavior of the sequence $(x^k)_{k \ge 0}$ that is defined by equations \eqref{eq:inertial_update} and \eqref{eq:QOP_intertial}. 
\subsection{Asymptotic Analysis}
The asymptotic analysis of the sequence $(x^k)_{k \ge 0}$ defined by \eqref{eq:inertial_update} and \eqref{eq:QOP_intertial} works surprisingly similar to the asymptotic analysis of the trajectories $x(t)$ of the differential equation \eqref{eq:IMOG'}. We start by proving that the sequence $(x^k)_{k \ge 0}$ satisfies a dissipative property. To this end, we introduce the following preparatory lemma.
\begin{mylemma}
Let $(x^k)_{k \ge 0}$ be defined by \eqref{eq:inertial_update} and \eqref{eq:QOP_intertial} with $x^0 = x^1 \in \H$ and $\alpha , h > 0$. Then, for all $i = 1,\dots, m$ it holds that
\begin{align*}
    \langle \nabla f_i(x^k) , x^{k+1} - x^k \rangle \le -\frac{\alpha}{h}\lVert x^{k+1} - x^k \rVert^2 + \frac{1}{2h^2} \left[\lVert x^k - x^{k-1} \rVert^2 - \lVert x^{k+1} - x^k \rVert^2 \right].
\end{align*}
\label{lem:energy_lemma}
\end{mylemma}
\begin{proof}
Using the variational characterization of the convex projection in the identity \eqref{eq:proj_rel_imog}, we get for all $i = 1,\dots, m$,
\begin{align*}
    \langle \alpha h (x^{k+1} - x^k) + h^2 \nabla f_i(x^k) + (x^{k+1} - x^k) - (x^k - x^{k-1}), \alpha h (x^{k+1} - x^k) \rangle \le 0,
\end{align*}
which can be rearranged into
\begin{align}
\label{eq:grad_step_esti}
    \begin{split}
    &\langle \nabla f_i(x^k) , x^{k+1} - x^k \rangle\\
    \le & -\left(\frac{1}{h^2} + \frac{\alpha}{h}\right) \lVert x^{k+1} - x^k \rVert + \frac{1}{h^2} \langle x^{k+1} - x^k, x^k - x^{k-1} \rangle\\
    \le & -\frac{\alpha}{h}\lVert x^{k+1} - x^k \rVert^2 + \frac{1}{2h^2} \left[\lVert x^k - x^{k-1} \rVert^2 - \lVert x^{k+1} - x^k \rVert^2 \right].
    \end{split}
\end{align}
\end{proof}
Using Lemma \ref{lem:energy_lemma}, we show that there exists an energy sequence which can be seen as a discretization of the energy function defined in Proposition \ref{prop:energy_estimation_const}.
\begin{myprop}
Assume that the gradients $\nabla f_i$ of the objective functions are globally $L$-Lipschitz continuous for all $i = 1,\dots,m$ and further assume $Lh < 2\alpha$. Then 
\begin{align}
    \E_{i,k} \coloneqq f_i(x^k) + \frac{1}{2h^2}\lVert x^{k} - x^{k-1} \rVert^2,
\label{eq:discrete_IMOG_energy}
\end{align}
is monotonically decreasing.
\label{prop:discrete_IMOG_energy}
\end{myprop}
\begin{proof}
We start with investigating the difference
\begin{align*}
    \E_{i,k+1} - \E_{i,k} = f_i(x^{k+1}) - f_i(x^k) + \frac{1}{2h^2} \left[ \lVert x^{k+1} - x^k \rVert - \lVert x^k - x^{k-1} \rVert \right].
\end{align*}
Using that $f_i$ is convex with $L$-Lipschitz continuous gradient, we estimate the expression above by
\begin{align*}
    \le \langle \nabla f_i(x^k), x^{k+1} - x^k \rangle + \frac{L}{2}\lVert x^{k+1} - x^k \rVert^2 + \frac{1}{2h^2} \left[ \lVert x^{k+1} - x^k \rVert - \lVert x^k - x^{k-1} \rVert \right].
\end{align*}
Using Lemma \ref{lem:energy_lemma}, we estimate this term by
\begin{align*}
    \le& \left( \frac{L}{2} - \frac{\alpha}{h} \right) \lVert x^{k+1} - x^k \rVert^2 - \frac{1}{2 h^2} \lVert x^{k+1} - 2x^k + x^{k-1} \rVert^2.
\end{align*}
For $hL < 2 \alpha$ it holds that $\left( \frac{L}{2} - \frac{\alpha}{h} \right) < 0$ and we get
\begin{align}
    \E_{i,k+1} - \E_{i,k} \le\left( \frac{L}{2} - \frac{\alpha}{h} \right) \lVert x^{k+1} - x^k \rVert^2 - \frac{1}{2 h^2} \lVert x^{k+1} - 2x^k + x^{k-1} \rVert^2,
    \label{eq:energy_decrease}
\end{align}
which completes the proof.
\end{proof}
The following corollary is an immediate consequence of Proposition \ref{prop:discrete_IMOG_energy}.
\begin{mycorollary}
Assume all conditions of Proposition \ref{prop:discrete_IMOG_energy} are met. Then, for all $i = 1,\dots, m$ and all $k \ge 1$ it holds that $f_i(x^k) \le f_i(x^0)$.
\label{cor:x_k_level_set_inertial}
\end{mycorollary}
Corollary \ref{cor:x_k_level_set_inertial} hints at a condition that guarantees that the sequence $(x^k)_{k\ge 0}$ remains bounded. If the level set $\mathcal{L}_i(f_i(x^0)) \coloneqq \lbrace x \in \H \,:\, f_i(x) \le f_i(x^0) \rbrace$ of one objective function $f_i$ is bounded, the sequence $(x^k)_{k \ge 0}$ remains bounded. In the following proposition we collect some immediate consequences of Proposition \ref{prop:discrete_IMOG_energy}.
\begin{myprop}
Assume the gradients $\nabla f_i$ of the objective functions are $L$-Lipschitz continuous on a bounded set containing the sequence $(x^k)_{k\ge 0}$, that is defined by equations \eqref{eq:inertial_update} and \eqref{eq:QOP_intertial}. Assume $Lh < 2\alpha$, then for all $i = 1,\dots, m$ the following statements hold.
\begin{enumerate}[i)]
    \item $\E_{i,k} \to \E_i^{\infty} > - \infty$ as $k \to +\infty$
    \item $\sum_{k = 0}^\infty \lVert x^{k+1} - x^k \rVert^2 < +\infty$
    \item $f_i(x^k) \to \E_i^{\infty}$ as $ k \to + \infty$
\end{enumerate}
\label{prop:discrete_IMOG_energy_properties}
\end{myprop}

\begin{proof}
i) \quad Proposition \ref{prop:discrete_IMOG_energy} states that $\E_{i,k}$ is monotonically decreasing. Therefore, $\E_{i,k} \to \E_{i}^{\infty}$ holds. We have to show that $\E_i^{\infty} > -\infty$. Since the objective functions $f_i$ have Lipschitz continuous gradients on a bounded set containing $(x^k)_{k\ge 0}$, it follows by the mean value theorem that $f_i$ is bounded on this sets and in particular on $(x^k)_{k \ge 0}$. Therefore, we conclude
\begin{align*}
    \E_{i}^{\infty} = \lim_{k \to +\infty} f_i(x^k) +  \frac{1}{2h^2} \lVert x^k - x^{k-1} \rVert^2 \ge \liminf_{k \to + \infty} f_i(x^k) > -\infty.
\end{align*}
ii) \quad From the inequality \eqref{eq:energy_decrease} we immediately follow
\begin{align}
\label{eq:finite_energy_sum}
\begin{split}
    &\E_{i, K+1} - \E_{i,1} = \sum_{k=1}^K \left( \E_{i, k+1} - \E_{i,k} \right) \\
    \le& \sum_{k=1}^K \left(\frac{L}{2} - \frac{\alpha}{h}\right) \lVert x^{k+1} - x^k \rVert^2  - \frac{1}{h^2} \sum_{k=1}^K \lVert  x^{k+1} - 2x^k + x^{k-1} \rVert^2 .
\end{split}
\end{align}
Since $Lh < 2\alpha$, it holds that $\left(\frac{\alpha}{h} - \frac{L}{2}\right) > 0$ and therefore we get for all $K \ge 1$
\begin{align*}
    \left(\frac{\alpha}{h} - \frac{L}{2}\right)\sum_{k=1}^K \lVert x^{k+1} - x^k \rVert^2 \le 
    \E_{i, 1} - \E_{i, K+1}.
\end{align*}
From part i), we know that the right hand side converges which completes the proof of ii).\\ \\
iii)\quad Since $\E_{i,k}  \to \E_i^{\infty}$ and $\lVert x^{k+1} - x^k \rVert^2 \to 0$, it follows that $f_i(x^k) \to \E_i^{\infty}$.
\end{proof}
We use the following discrete version of Opial's Lemma to prove that $(x^k)_{k\ge 0}$ converges weakly to a Pareto critical point of \eqref{eq:MOP}. 
\begin{mylemma}[Opial's Lemma \cite{Opial1967}]
Let $S \subset \H$ be nonempty and let $(x_k)_{k\ge 0}$ be a sequence in $\H$ that satisfies the following conditions.
\begin{enumerate}[i)]
    \item For all $z \in S$ $\lim_{k \to + \infty} \lVert x^k - z\rVert $ exists.
    \item Every weak sequential cluster point of $(x_k)_{k\ge 0}$ belongs to $S$.
\end{enumerate}
Then, it follows that $(x_k)_{k\ge 0}$ converges weakly to an element in $S$.
\label{lem:opial}
\end{mylemma}
We will use Opial's Lemma on the set
\begin{align}
\label{def:discrete_S}
    S \coloneqq \left\lbrace z \in \H \,:\, f_i(z) \le \lim_{k \to + \infty}f_i(x_k) \,\,\text{for all  } i = 1,\dots, m,\right\rbrace.
\end{align}

\begin{theorem}
Assume the gradients $\nabla f_i$ of the objective functions are $L$-Lipschitz continuous on a bounded set containing the sequence $(x^k)_{k\ge 0}$, defined by \eqref{eq:inertial_update} and \eqref{eq:QOP_intertial} and further assume $Lh < 2\alpha$. Then, $(x^k)_{k \ge 0}$ converges weakly to a Pareto critical point of \eqref{eq:MOP}.
\end{theorem}
\begin{proof}
We show that $(x^k)_{k \ge 0}$ satisfies Opial's Lemma for the set $S$ defined by \eqref{def:discrete_S}. We start by showing a quasi Fejér property of the sequence $(x_k)_{k\ge 0}$. For a fixed $z \in S$, define the sequence
\begin{align*}
    h_k \coloneqq \frac{1}{2}\lVert x^{k} - z \rVert^2.
\end{align*}
It is easy to check that 
\begin{align*}
    h_{k+1} &= h_k + \langle x^{k+1} - x^k, x^k - z \rangle + \frac{1}{2}\lVert x^{k+1} - x^k \rVert^2.
\end{align*}
Proposition \ref{prop:discrete_IMOG_energy_properties} guarantees the monotonicity of $\E_{i,k}$. Since $z \in S$, from the convexity of $f_i$ we can deduce for all $i = 1,\dots, m$ that
\begin{align*}
    &\E_{i,k} = f_i(x^k) + \frac{1}{2h^2} \lVert x^k - x^{k-1} \rVert^2 \ge \E_i^{\infty} \ge f_i(z) \ge f_i(x^k) + \langle \nabla f_i(x^k), z - x^k \rangle,
\end{align*}
and therefore
\begin{align*}
    \langle \sum_{i=1}^m \theta_i\nabla f_i(x^k), z - x^k \rangle \le \frac{1}{2h^2} \lVert x^k - x^{k-1} \rVert^2.
\end{align*}
Using this inequality we can show
\begin{align*}
    &\frac{h^2}{1+\alpha h} \langle \sum_{i=1}^m \theta_i^k \nabla f_i(x^k) , z - x^k \rangle = \langle x^k - x^{k+1} - \frac{1}{1+\alpha h}(x^k - x^{k-1}), z - x^k \rangle\\
    = &\langle x^{k+1} - x^k , x^k - z \rangle - \frac{1}{1+\alpha h} \langle x^{k-1} - x^k, x^k - z \rangle \le \frac{1}{2(1 + \alpha h)}\lVert x^k - x^{k-1} \rVert^2,
\end{align*}
which leads to the inequality
\begin{align*}
    \langle x^{k+1} - x^k , x^k - z \rangle \le \frac{1}{1+\alpha h} \langle x^{k-1} - x^k, x^k - z \rangle + \frac{1}{2(1 + \alpha h)}\lVert x^k - x^{k-1} \rVert^2.
\end{align*}
We use this inequality to show
\begin{align*}
    h_{k+1} - h_k = &\langle x^{k+1} - x^k, x^k - z \rangle + \frac{1}{2}\lVert x^{k+1} - x^k \rVert^2\\
    \le &\frac{1}{1+\alpha h} \langle x^{k-1} - x^k, x^k - z \rangle + \frac{1}{2}\lVert x^{k+1} - x^k \rVert^2 + \frac{1}{2(1 + \alpha h)}\lVert x^k - x^{k-1} \rVert^2\\
    = & \frac{1}{1+\alpha h} \left[ h_{k} - h_{k-1}  + \frac{1}{2}\lVert x^k - x^{k-1} \rVert^2 \right] + \frac{1}{2}\lVert x^{k+1} - x^k \rVert^2\\
    &+ \frac{1}{2(1 + \alpha h)} \lVert x^k - x^{k-1} \rVert^2\\
    \le & \frac{1}{1+\alpha h} (h_{k} - h_{k-1}) + \frac{1}{2}\lVert x^{k+1} - x^k \rVert^2 + \frac{1}{1 + \alpha h}\lVert x^k - x^{k-1} \rVert^2.
\end{align*}
Defining $\theta_k \coloneqq h_{k+1} - h_k$, $\delta_k \coloneqq \frac{1}{1+\alpha h}\lVert x^{k} - x^{k-1} \rVert^2 + \frac{1}{2}\lVert x^{k+1} - x^k \rVert^2$ and $a \coloneqq \frac{1}{1+\alpha h}$, we can therefore conclude
\begin{align*}
    \theta_{k+1} \le a \theta_k + \delta_k.
\end{align*}
Proposition \ref{prop:discrete_IMOG_energy_properties} states that $\sum_{k=1}^\infty \delta_k < +\infty$. Therefore, we can use Theorem 2.1 in \cite{Alvarez2001} or Theorem 3.1 in \cite{Alvarez2000} to show that $h_k$ converges. To use Opial's Lemma, we also have to show that all weak sequential cluster points of $(x_k)_{k\ge 0}$ belong to $S$. Since the sequence $(x_k)_{k\ge 0}$ is bounded, it possesses at least one sequential cluster point that we denote by $x^{\infty}$ and a subsequence $(x_{k_l})_{l\ge 0}$ that converges weakly to $x^{\infty}$. Since $f_i$ is convex and continuous, it is also weakly lower semicontinuous and it follows that for all $i = 1,\dots, m$
\begin{align*}
    f_i(x^{\infty}) \le \liminf_{l \to + \infty} f_i(x_{k_l}) = \lim_{k\to +\infty} f_i(x_k),
\end{align*}
where the equality follows from the fact that the limit exists. Therefore, $x^{\infty} \in S$ and hence $S$ is nonempty. Then, Opial's Lemma \ref{lem:opial} states that $(x^k)_{k \ge 0}$ converges weakly to an element in $S$ that we denote by $x^{\infty}$. We will show that each weak sequential cluster point of $(x^k)_{k \ge 0}$ is Pareto critical. By the definition of the sequence $(x^k)_{k \ge 0}$ in \eqref{eq:inertial_update}, it holds that
\begin{align*}
    \sum_{k=1}^{\infty}\left\lVert \sum_{i=1}^m\theta_i^k\nabla f_i(x^k) \right\rVert^2
    &= \sum_{k=1}^{\infty}\left\lVert \frac{1 + \alpha h}{h^2}(x^{k+1} - x^{k}) + \frac{1}{h^2}(x^k - x^{k-1}) \right\rVert^2.
\end{align*}
This sum is finite by part ii) of Proposition \ref{prop:discrete_IMOG_energy_properties}. Thus, we know that the sequence $g^k \coloneqq \sum_{i=1}^m\theta_i^k\nabla f_i(x^k) \in \co(\nabla f_i(x^k))$ converges strongly to zero. Since $x^k$ converges weakly to $x^{\infty}$, Lemma \ref{lem:demiclosedness_co_grad_f} states that $0 \in \co(\nabla f_i(x^{\infty}))$ and hence $x^{\infty}$ is Pareto critical. 
\end{proof}
\section{An Accelerated Multiobjective Gradient Method}
\label{sec:acc_multi_grad_algo}
In this section, we define a multiobjective gradient method with Nesterov acceleration based on the inertial method we discussed in the previous subsection.
\subsection{The single objective Case}
In this subsection we present Nesterov's method in the single objective setting and point out its relation to an intertial gradient-like dynamical system with asymptotically vanishing damping. Consider the problem
\begin{align*}
    \min_{x \in \H} f(x),
\end{align*}
where $f:\H \to \R$ is convex and differentiable with $L$-Lipschitz continuous gradient $\nabla f(x)$. For $\alpha \ge 3$, $0 < s \le \frac{1}{L}$ and $x^0, x^1 \in \H$, define the sequence $(x^k)_{k \ge 0}$ by
\begin{align}
\left.
\begin{array}{rl}
    y^k &= x^k + \frac{k -1}{k + \alpha - 1}(x^k - x^{k-1}),\\
    x^{k+1} &= y^k - s \nabla f(y^k)
\end{array}
    \right\rbrace \,\,\text{for}\,\, k \ge 1.
\label{eq:Nesterov_method_single objective}
\end{align}
If $\argmin f \neq \emptyset$ it can be shown that $f(x^k) - \min_{x\in \H} f(x) = \mathcal{O}(k^{-2})$ and that $\lVert x^{k+1} - x^k \rVert = \mathcal{O}(k^{-1})$. For $\alpha > 3$, it holds that $f(x^k) - \min_{x\in \H} f(x) = o(k^{-2})$, $\lVert x^{k+1} - x^k \rVert = o(k^{-1})$ and that $(x^k)_{k \ge 0}$ converges weakly to an element in $\argmin f$  \cite{Attouch2015_2}. Nesterov's method is related to the following gradient-like dynamical system with asymptotically vanishing damping
\begin{align}
    \ddot{x}(t) + \frac{\alpha}{t} \dot{x}(t) + \nabla f(x(t)) = 0.
    \label{eq:AVD_Single_objective}
\end{align}
The algorithm \eqref{eq:Nesterov_method_single objective} can be derived as a discretization of \eqref{eq:AVD_Single_objective}. This relation is further investigated in \cite{Su2014, Attouch2022}.
\subsection{Introducing Nesterov Acceleration in \eqref{eq:IMOG'}}
\label{subsec:intro_new_acc_imog'}
We formally define the following gradient-like system with asymptotically vanishing damping for multiobjective optimization. 
\begin{align}
\ddot{x}(t) + \frac{3}{t}\dot{x}(t) + \proj_{C(x(t))}{(-\ddot{x}(t))} = 0.
\label{eq:AVD_Single_objective_a=3}
\end{align}
We will not analyze this system in detail but leave this for future work. We show that an implicit discretization of this system leads to an accelerated multiobjective gradient method with an improvement convergence rate of the function values. We equivalently write \eqref{eq:AVD_Single_objective_a=3} as
\begin{align}
    \frac{3}{t}\dot{x}(t) + \proj_{C(x(t)) + \ddot{x}(t)}(0) = 0.
\label{eq:AVD_Single_objective_a=3_rewritten}
\end{align}
Using the same Ansatz as in Section 2 of \cite{Su2014}, we show that we can derive the differential equation \eqref{eq:AVD_Single_objective_a=3} from the scheme 
\begin{align}
    \frac{3}{k} (x^{k+1} - x^k) + \proj_{sC(y^k) + (x^{k+1} - 2x^k + x^{k-1})}(0) = 0,
    \label{eq:finite_diff_scheme_a_3}
\end{align}
with $y^k = x^k + \frac{k-1}{k+2}(x^k - x^{k-1})$. We divide \eqref{eq:finite_diff_scheme_a_3} by $\sqrt{s}$ to get
\begin{align}
    \frac{3}{k} \frac{x^{k+1} - x^k}{\sqrt{s}} + \proj_{\sqrt{s}C(y^k) + \frac{x^{k+1} - 2x^k + x^{k-1}}{\sqrt{s}}}(0) = 0.
    \label{eq:finite_diff_scheme_a_3_normalized}
\end{align}
We use the Ansatz $x_k \approx x(k\sqrt{s})$ for some smooth curve $x(t)$ defined for all $t \ge 0$. Write $k = \frac{t}{\sqrt{s}}$. When the step size $s$ goes to zero $X(t) \approx x_{\frac{t}{\sqrt{s}}} = x_k$ and  $X(t) \approx x_{\frac{t + \sqrt{s}}{\sqrt{s}}} = x_{k+1}$. Then, Taylor expansion gives
\begin{align}
    \frac{x^{k+1} - x^k}{\sqrt{s}} = \dot{x}(t) + \frac{1}{2}\ddot{x}(t)\sqrt{s} + o(\sqrt{s}),\quad \frac{x^{k} - x^{k-1}}{\sqrt{s}} = \dot{x}(t) - \frac{1}{2}\ddot{x}(t)\sqrt{s} + o(\sqrt{s}),
\label{eq:scheme_1_deriv}
\end{align}
and hence
\begin{align}
    \frac{x^{k+1} - 2x^k +x^{k-1}}{\sqrt{s}} = \ddot{x}(t)\sqrt{s} + o(\sqrt{s}).
\label{eq:scheme_2_deriv}
\end{align}
For all $i = 1,\dots, m$, we have $\sqrt{s}\nabla f_i(y^k) = \sqrt{s}\nabla f_i(x(t)) + o(\sqrt{s})$. Since the convex projection depends in a well-behaved manner on the convex set we project onto, we get
\begin{align}
    \proj_{\sqrt{s}C(y^k) + \frac{x^{k+1} - 2x^k + x^{k-1}}{\sqrt{s}}}(0) = \sqrt{s} \proj_{C(x(t)) + \ddot{x}(t)}(0) + o(\sqrt{s}).
\label{eq:projection_approx}
\end{align}
Combining \eqref{eq:scheme_1_deriv}, \eqref{eq:scheme_2_deriv} and \eqref{eq:projection_approx}, we get from \eqref{eq:finite_diff_scheme_a_3_normalized} 
\begin{align*}
    \frac{3\sqrt{s}}{t}\left( \dot{x}(t) + \frac{1}{2}\ddot{x}(t) \sqrt{s} + o(\sqrt{s}) \right) + \sqrt{s} \proj_{C(x(t)) + \ddot{x}(t)}(0) + o(\sqrt{s}) = 0.
\end{align*}
Comparing the coefficients of $\sqrt{s}$, we obtain
\begin{align*}
    \frac{3}{t}\dot{x}(t) + \proj_{C(x(t)) + \ddot{x}(t)}(0) = 0.
\end{align*}
We have shown that the differential equation \eqref{eq:AVD_Single_objective_a=3} can be derived from the scheme \eqref{eq:finite_diff_scheme_a_3}. Using Lemma \ref{lem:proj_1} on \eqref{eq:finite_diff_scheme_a_3}, we get that $x^{k+1}$ is uniquely defined as
\begin{align*}
    x^{k+1} & = - \left(\frac{k}{k + 3} \proj_{s C(y^k) - 2x^k + x^{k-1}}(-x^k) - \frac{3}{k + 3}x^k \right) \\
            & = x^k - \frac{k}{k + 3} \proj_{s C(y^k) - x^k - x^{k-1}}(0).
\end{align*}
The last term can be witten as
\begin{align*}
    x^k + \frac{k}{k + 3}(x^k - x^{k-1}) - \frac{sk}{k + 3} \sum_{i = 1}^m \theta_i^k \nabla f_i(y^k),
\end{align*}
where $\theta^k \in \R^m$ is a solution to the quadratic optimization problem 
\begin{align}
\begin{split}
    \min_{\theta \in \R^m} \, \left\lVert s \left(\sum_{i=1}^m \theta_i \nabla f_i(y^k) \right) - (x^k - x^{k-1}) \right\rVert^2\,\,\st \, \theta \ge 0 \,\,\text{and}\,\, \sum_{i = 1}^m \theta_i = 1.
\end{split}
\label{eq:QOP_a_3_first}
\end{align}
We want to drop the factor $\frac{k}{k + 3}$ in front of the term $\sum_{i = 1}^m \theta_i^k \nabla f_i(y^k)$ to get a method that more closely resembles \eqref{eq:Nesterov_method_single objective}. In addition, we perform a shift of the index $k$ to transform $\frac{k}{k + 3}$ into $\frac{k-1}{k + 2}$. The final method we define in this subsection can then be defined as follows. Let $x^0 = x^1 \in \H$ and $s > 0$. Define the scheme
\begin{align}
\begin{split}
    \left.
    \begin{array}{c}
    y^{k} = x^k + \frac{k-1}{k+2}(x^k - x^{k-1}),\\
    x^{k+1} = y^k - s\sum_{i = 1}^m \theta_i^k \nabla f_i(y^k)
    \end{array}
    \right\rbrace \,\,\text{ for }\,\,k \ge 1,
\end{split}
    \label{eq:Acc_MOO_desc_a=3}
\end{align}
where in each step $\theta^k \in \R^m$ is a solution to the quadratic optimization problem
\begin{align}
\begin{split}
    \min_{\theta \in \R^m} \, \left\lVert s \left(\sum_{i=1}^m \theta_i \nabla f_i(y^k) \right) - \frac{k - 1}{k + 2}(x^k - x^{k-1}) \right\rVert^2\,\, \st \, \theta \ge 0 \,\,\text{and}\,\, \sum_{i = 1}^m \theta_i = 1.
\end{split}
\label{eq:QOP_a_3}
\end{align}
The fact that we have to transform the quadratic optimization problem from \eqref{eq:QOP_a_3_first} into \eqref{eq:QOP_a_3} is an observation from the proof of Proposition \ref{prop:energy_acc}. The presented method is still asymptotically equivalent to the scheme defined by \eqref{eq:finite_diff_scheme_a_3}. We summarize the defined method in Algorithm \ref{algo:ACC_GRAD} for later references.
\begin{algorithm} 
	\caption{Accelerated multiobjective gradient method}
	\label{algo:ACC_GRAD}
	\begin{algorithmic}[1] 
		\Require Choose $x^0 = x^1 \in \H$, $0 < s \le \frac{1}{L}$ and set $k = 1$.
		\State Set $y^k = x^k + \frac{k-1}{k+2}(x^k - x^{k-1})$.
		\State Compute $\theta^k \in \R^m$ by solving
		\begin{align*}
		\begin{split}
            \min_{\theta \in \R^m} \, \left\lVert s \left(\sum_{i=1}^m \theta_i \nabla f_i(y^k) \right) - \frac{k - 1}{k + 2}(x^k - x^{k-1}) \right\rVert^2\,\, \st \, \theta \ge 0\,\, \text{and}\,\, \sum_{i = 1}^m \theta_i = 1.
        \end{split}
        \end{align*}
		\State Set $x^{k+1} = y^k - s\sum_{i=1}^m \theta_i^k \nabla f_i(y^k)$
		\If{stopping condition is true}
            \State Stop.
        \Else
            \State Update $k \leftarrow k+1$ and go to step 1.
        \EndIf
	\end{algorithmic} 
\end{algorithm}
\subsection{A Dissipative Property}
We start our investigations of Algorithm \ref{algo:ACC_GRAD} with an energy estimate analogous to Proposition \ref{prop:discrete_IMOG_energy} for the inertial method.
\begin{myprop}
Assume that the gradients $\nabla f_i$ of the objective functions are globally $L$-Lipschitz continuous for all $i = 1,\dots, m$ and further assume $sL \le 1$. Define for all $k \ge 1$ the energy sequence
\begin{align*}
    \E_{i,k} \coloneqq f_i(x^k) + \frac{1}{2s}\lVert x^{k} - x^{k-1}\rVert^2.
\end{align*}
For all $k \ge 1$, it holds that
\begin{align*}
    \E_{i,k+1} - \E_{i,k} \le - \frac{1}{2s}\frac{3}{k + 2}\lVert x^k - x^{k-1} \rVert^2.
\end{align*}
\label{prop:energy_acc}
\end{myprop}
\begin{proof}
From the definition of $x^k$ and $y^k$ in \eqref{eq:Acc_MOO_desc_a=3} we get
\begin{align*}
    x^{k+1} - x^k + \proj_{s C(y^k) - \frac{k-1}{k + 2}(x^k - x^{k-1})}(0) &=0.
\end{align*}
Hence, for all $i = 1, \dots, m$ it holds that
\begin{align*}
    \langle x^{k+1} - x^k + s \nabla f_i(y^k) - \frac{k-1}{k + 2}(x^k - x^{k-1}), x^{k+1} - x^k \rangle \le 0,
\end{align*}
from which we follow
\begin{align*}
    &s \langle \nabla f_i(y^k), x^{k+1} - x^k \rangle \le - \lVert x^{k+1} - x^k \rVert^2 + \frac{k-1}{k + 2}\langle x^{k+1} - x^k, x^k - x^{k-1} \rangle \\
    =&-\frac{3}{k+2}\lVert x^{k+1} - x^k \rVert  - \frac{
    1}{2}\frac{k - 1}{k + 2} \lVert x^{k+1} - 2x^k + x^{k-1} \rVert^2 \\
    & + \frac{1}{2} \frac{k - 1}{k + 2} \left[ \lVert  x^{k} - x^{k-1}  \rVert - \lVert x^{k+1} - x^k \rVert \right].
\end{align*}
Writing out the definition of $y^k$, one can easily verify that
\begin{align*}
    \lVert x^{k+1} - y^k \rVert^2 \le \frac{k-1}{k + 2 } \lVert x^{k+1} - 2x^k + x^{k-1} \rVert^2 + \frac{3}{k + 2}\lVert x^{k+1} - x^k \rVert^2.
\end{align*}
Combining the inequalities above and using $sL \le 1$ we get
\begin{align*}
    & s(f_i(x^{k+1}) - f_i(x^k)) \le s \langle \nabla f_i(y^k), x^{k+1} - x^k \rangle + \frac{1}{2}\lVert x^{k+1} - y^k \rVert^2 \\
    \le& -\frac{1}{2}\frac{3}{k + 2} \lVert x^{k+1} - x^k \rVert^2 + \frac{1}{2} \frac{k-1}{k + 2} \left[ \lVert x^k - x^{k-1} \rVert - \lVert x^{k+1} - x^k\rVert\right]\\
    = & \frac{1}{2} \left[ \lVert x^k - x^{k-1} \rVert^2 - \lVert x^{k+1} - x^k\rVert^2\right] - \frac{1}{2}\frac{3}{k + 2}\lVert x^k - x^{k-1} \rVert^2,
\end{align*}
which completes the proof.
\end{proof}
\begin{mycorollary}
Let $(x^k)_{k \ge 0}$ be a sequence defined by \eqref{eq:Acc_MOO_desc_a=3}. Then, it holds that for all $k \ge 0$ and all $i = 1,\dots, m$
\begin{align*}
    f_i(x^k) \le f_i(x^0).
\end{align*}
\end{mycorollary}
\subsection{Convergence of Function Values with Rate $\mathcal{O}(k^{-2})$}
The proof in this section relies on the proof by Fukuda, Tanabe and Yamashita \cite{Tanabe2022_2} for their accelerated gradient method and the proof of Attouch and Peypouquet \cite{Attouch2015_2} for the single objective case. The following definition is aligned with \cite{Tanabe2020} and the concept of merit functions that was introduced in \cite{Tanabe2020} and further utilized in \cite{Tanabe2022, Tanabe2022_2}. For $z \in \H$ define
\begin{align*}
    \sigma_k(z) \coloneqq \min_{i=1,\dots, m} f_i(x^k) - f_i(z).
\end{align*}
\begin{mylemma}
It holds that
\begin{align*}
    \sigma_{k+1}(z) &\le -\frac{1}{s}\langle x^{k+1} - y^k , y^k - z\rangle - \frac{1}{2s}\lVert x^{k+1} - y^k\rVert^2.
\end{align*}
\label{lem:sigma_k+1}
\end{mylemma}
\begin{proof}
The objective functions $f_i$ are convex with $L$-Lipschitz continuous gradients. Therefore, for all $i = 1,\dots, m$ it holds that
\begin{align}
\label{eq:lemma1_1}
\begin{split}
    & f_i(x^{k+1}) - f_i(z) \le f_i(x^{k+1}) - f_i(y^k) + f_i(y^k) - f_i(z)\\
    \le &\langle \nabla f_i(y^k), x^{k+1} - y^k \rangle + \frac{L}{2} \lVert x^{k+1} -y^k \rVert^2 + \langle \nabla f_i(y^k), y^k - z \rangle.
    \end{split}
\end{align}
The definition of $\sigma_k(z)$ gives
\begin{align}
\label{eq:lemma1_2}
    \sigma_{k+1}(z) = \min_{i=1,\dots, m} f_i(x^{k+1}) - f_i(z) \le \sum_{i=1}^m \theta_i^k \left( f_i(x^{k+1}) - f_i(z) \right).
\end{align}
Combining \eqref{eq:lemma1_1} and \eqref{eq:lemma1_2} and using $\sum_{i=1}^m \theta_i^k \nabla f_i(y^k) = \frac{1}{s}(y^k - x^{k+1})$ we get the desired inequality.
\end{proof}
We want to find a similar inequality for the expression $f_i(x^{k+1}) - f_i(x^k)$. To this end, we introduce the following lemma.
\begin{mylemma}
Define the optimization problem 
\begin{align}
\label{eq:primal_problem}
\begin{split}
    \min_{(v, \alpha) \in \H \times \R} &\Phi(v, \alpha) \coloneqq \frac{1}{2}\lVert sv + (y^k - x^k) \rVert^2 + \alpha\\
    \st\,\, &g_i(v,\alpha) \coloneqq \langle s\nabla f_i(y^k) - (y^k - x^k), sv + (y^k - x^k) \rangle - \alpha \le 0.
\end{split}
\end{align}
Then, it holds that the dual problem to this problem is the quadratic problem \eqref{eq:QOP_a_3}. An optimal solution $\theta^*$ to \eqref{eq:QOP_a_3} satisfies
\begin{align*}
    \langle s\sum_{i=1}^m \theta_i^* \nabla f_i(y^k), x^{k+1} - x^k\rangle = \max_{i=1,\dots,m} \langle s\nabla f_i(y^k), x^{k+1} - x^k\rangle.
\end{align*}
\label{lem:primal_dual_comp}
\end{mylemma}
\begin{proof}
Since $\mathcal{H}$ is potentially infinite-dimensional, we need duality statements for infinite-dimensional constrained optimization problems. The statements we use in this proof can be found in Sections 8.3 to 8.6 of \cite{Luenberger1997}. Since the optimization problem \eqref{eq:primal_problem} has a fairly simple structure, we will not recite every result we use. The duality between \eqref{eq:primal_problem} and \eqref{eq:QOP_a_3} follows from a straightforward computation. Since the objective function $\Phi(v, \alpha)$ of \eqref{eq:primal_problem} is convex and all constraints $g_i(v, \alpha)$ are linear, strong duality holds. Hence a KKT point $((v^*, \alpha^*), \theta^*) \in (\H \times \R) \times \R^m$ of problem \eqref{eq:primal_problem} yields a solution to \eqref{eq:QOP_a_3}. From the KKT conditions for \eqref{eq:primal_problem} we get that
\begin{align*}
    v^* = -s\sum_{i=1}^m \theta_i^* \nabla f_i(y^k).
\end{align*}
For all $i = 1,\dots, m$ it holds that $g_i(v,\alpha) \le 0$ and hence
\begin{align*}
    \langle s \nabla f_i(y^k) - (y^k - x^k), sv + (y^k - x^k)\rangle \le \alpha.
\end{align*}
By the complementarity of $\theta_i^*$ and $g_i(v^*, \alpha^*)$ we get
\begin{align*}
    &\langle s \sum_{i=1}^m \theta_i^* \nabla f_i(y^k) - (y^k - x^k), sv^* + (y^k - x^k) \rangle = \alpha^*\\
    = &\max_{i=1,\dots, m} \langle s\nabla f_i(y^k) - (y^k - x^k), sv^* + (y^k - x^k) \rangle.
\end{align*}
The second equality above follows from the fact that  $\theta_j^* > 0$ holds for at least one $j \in \{1,\dots, m\}$ as a consequence of the dual feasibility.\\
Using $v^* = -\sum_{i=1}^m \theta_i^* \nabla f_i(y^k)$, we get $sv^* = x^{k+1} - y^k$ and therefore \begin{align*}
    &\langle s \sum_{i=1}^m \theta_i^* \nabla f_i(y^k) - (y^k - x^k), x^{k+1} - x^k \rangle\\
    = &\max_{i=1,\dots, m} \langle s\nabla f_i(y^k) - (y^k - x^k), x^{k+1} - x^k) \rangle.
\end{align*}
\end{proof}
\begin{mylemma}
    \begin{align*}
        \sigma_{k+1}(z)- \sigma_k(z) &\le -\frac{1}{s}\langle x^{k+1} - y^k , y^k - x^k\rangle - \frac{1}{2s} \lVert x^{k+1} - y^k\rVert^2.
    \end{align*}
\label{lem:sigma_k+1-sigma_k}
\end{mylemma}
\begin{proof}
For all $a, b \in \R^m$ it holds that 
\begin{align*}
\left(\min_{i=1, \dots, m} a_i\right) - \left(\min_{i=1,\dots, m} b_i \right) \le \max_{i=1,\dots, m} \left(a_i - b_i\right)    
\end{align*}
and therefore for all $z \in \H$
\begin{align*}
    \sigma_{k+1}(z)- \sigma_k(z) \le \max_{i=1,\dots, m} \left( f_i(x^{k+1}) - f_i(x^k) \right).
\end{align*}
Using that the objective functions $f_i$ are convex with $L$-Lipschitz continuous gradients and the fact that $sL \le 1$, we can bound this expression by
\begin{align*}
    \le \max_{i=1,\dots,m} \left(\langle \nabla f_i(y^k), x^{k+1} - x^k \rangle + \frac{1}{2s} \lVert x^{k+1} - y^k \rVert^2\right).
\end{align*}
Now we use Lemma \ref{lem:primal_dual_comp} and get the equality
\begin{align*}
    = \sum_{i=1}^m \theta_i^k \langle \nabla f_i(y^k), x^{k+1} - x^k \rangle + \frac{1}{2s} \lVert x^{k+1} - y^k \rVert^2.
\end{align*}
From here, we continue by using the definitions of $x^k$ and $y^k$ from \eqref{eq:Acc_MOO_desc_a=3} to get
\begin{align*}
    =& \frac{1}{s} \langle y^k - x^{k+1}, x^{k+1} - x^k \rangle + \frac{1}{2s} \lVert x^{k+1} - y^k \rVert^2\\
    =& -\frac{1}{s} \langle x^{k+1} - y^{k}, y^{k} - x^k \rangle - \frac{1}{2s} \lVert x^{k+1} - y^k \rVert^2.\\
\end{align*}
\end{proof}
\begin{theorem}
The sequence $(x^k)_{k \ge 0}$ defined by \eqref{eq:Acc_MOO_desc_a=3} satisfies
\begin{align*}
    \sigma_{k}(z) \le \frac{2 \left( \lVert x^1- z \rVert^2 + \lVert x^2 - z \rVert^2 \right)}{s (k+1)^2 }. 
\end{align*}
\label{thm:convergence_sigma_k_z}
\end{theorem}
\begin{proof}
Lemma \ref{lem:sigma_k+1} and Lemma \ref{lem:sigma_k+1-sigma_k} state 
\begin{align*}
        \sigma_{k+1}(z) &\le -\frac{1}{s}\langle x^{k+1} - y^k , y^k - z\rangle - \frac{1}{2s}\lVert x^{k+1} - y^k\rVert^2 \,\,\text{  and  }\\
        \sigma_{k+1}(z)- \sigma_k(z) &\le -\frac{1}{s}\langle x^{k+1} - y^k , y^k - x^k\rangle - \frac{1}{2s} \lVert x^{k+1} - y^k\rVert^2.
    \end{align*}
Taking a convex combination of the last inequalities with weights $\frac{2}{k + 2}$ and $\frac{k}{k + 2}$ yields
\begin{align}
\begin{split}
    & \sigma_{k+1}(z) - \frac{k}{k + 2}\sigma_k(z) \\
    \le & -\frac{1}{s} \langle x^{k+1} - y^k , y^k -\frac{k}{k + 2}x^k - \frac{2}{k + 2}z \rangle - \frac{1}{2s} \lVert x^{k+1} - y^k \rVert^2\\
    = & \frac{1}{s} \langle x^{k+1} - y^k , \frac{k}{k + 2}(x^k - y^k) + \frac{2}{k + 2}(z - y^k) \rangle - \frac{1}{2s} \lVert x^{k+1} - y^k \rVert^2.
\end{split}
\label{eq:sigam_k+1_z_ineq}
\end{align}
Define
\begin{align}
    z^k \coloneqq \frac{k + 2}{2}y^k - \frac{k}{2}x^k = x^k + \frac{k-1}{2}(x^k - x^{k-1}),
    \label{eq:def_z^k}
\end{align}
and notice that
\begin{align}
\label{eq:z^k_identity}
    \frac{k}{k + 2}(y^k - x^k) + \frac{2}{k + 2}(y^k - z) = \frac{2}{k + 2}(z^k - z).
\end{align}
Using the identity \eqref{eq:z^k_identity} in \eqref{eq:sigam_k+1_z_ineq} we get
\begin{align}
    \sigma_{k+1}(z) \le \frac{k}{k + 2} \sigma_k(z) - \frac{2}{s(k + 2)}\langle x^{k+1} - y^k, z^k - z \rangle - \frac{1}{2s}\lVert x^{k+1} - y^k \rVert^2.
    \label{eq:sigma_k+1_inequality}
\end{align}
From the definition of $z^k$ in \eqref{eq:def_z^k} one can see that
\begin{align*}
    z^{k+1} = z^k + \frac{k+2}{2}(x^{k+1} - y^k).
\end{align*}
Using this identity we can simply compute the squared norm of $\lVert z^{k+1} - z \rVert^2$ as
\begin{align*}
    \lVert z^{k+1} - z \rVert^2 = \lVert z^k - z \rVert^2 + (k+2)\langle z^k - z, x^{k+1} - y^k \rangle + \left( \frac{k+2}{2} \right)^2 \lVert x^{k+1} - y^k \rVert^2.
\end{align*}
Rearranging this identity and multiplying with $\frac{2}{s(k+2)^2}$ yields
\begin{align}
\begin{split}
    &\frac{2}{s(k + 2)^2} \left(\lVert z^{k} - z \rVert^2 - \lVert z^{k+1} - z \rVert^2 \right)\\
    = &-\frac{2}{s(k + 2)}\langle z^k - z, x^{k+1} - y^k \rangle - \frac{1}{2s} \lVert x^{k+1} - y^k \rVert^2.
\end{split}
\label{eq:identity_norm_z^k_z}
\end{align}
Combining \eqref{eq:sigma_k+1_inequality} and \eqref{eq:identity_norm_z^k_z}, in total we get
\begin{align*}
    \sigma_{k+1}(z) \le \frac{k}{k + 2} \sigma_k(z) + \frac{4}{2s(k + 2)^2} \left(\lVert z^{k} - z \rVert^2 - \lVert z^{k+1} - z \rVert^2 \right).
\end{align*}
Multiplying both sides with $(k+2)^2$ then yields
\begin{align*}
    (k+2)^2\sigma_{k+1}(z) \le k(k+2) \sigma_k(z) + \frac{2}{s} \left(\lVert z^{k} - z \rVert^2 - \lVert z^{k+1} - z \rVert^2 \right).
\end{align*}
Using $k(k+2) \le (k+1)^2$ we get
\begin{align*}
    (k+2)^2\sigma_{k+1}(z) - (k+1)^2 \sigma_k(z) \le \frac{2}{s} \left(\lVert z^{k} - z \rVert^2 - \lVert z^{k+1} - z \rVert^2 \right).
\end{align*}
Summing this inequality from $k = 1, \dots, K$, we get for all $z \in \H$
\begin{align*}
    (K+2)^2\sigma_{K+1}(z) \le \frac{2}{s} \lVert x^1 - z \rVert^2 + 4 \sigma_1(z).
\end{align*}
Similar computations to Lemma \ref{lem:sigma_k+1} yield
\begin{align*}
    \sigma_1(z) \le \frac{1}{2s}\lVert x^2 - z \rVert^2 - \frac{1}{2s} \lVert x^2 - x^1 \rVert^2.
\end{align*}
Then, for all $k \ge 1$, we obtain
\begin{align*}
    \sigma_{k}(z) \le \frac{2 \left( \lVert x^1- z \rVert^2 + \lVert x^2 - z \rVert^2 \right)}{s (k+1)^2 }.
\end{align*}
\end{proof}
The theorem above is not straight forward to interpret since we only get convergence of order $\mathcal{O}(k^{-2})$ for $\min_{i=1,\dots, m} f_i(x^k) - f_i(z)$. This on it's own does not state that the vector $f(x^k) = \left( f_1(x^k), \dots ,f_m(x^k) \right)$ converges to an element of the Pareto front. However we can refine the statement of Theorem \ref{thm:convergence_sigma_k_z} in the following way to get a stronger convergence statement under a weak additional assumption.
\begin{theorem}
Assume in addition to the assumption in Theorem \ref{thm:convergence_sigma_k_z} that for all $x \in \mathcal{L}(f(x_0))$ there exists an $x^* \in \mathcal{L}^* \coloneqq P_w \cap \mathcal{L}(f(x_0))$ with $f(x^*) \le f(x)$ and
\begin{align}
    \sup_{f^* \in f(\mathcal{L}^*)} \inf_{x \in f^{-1}(\lbrace f^* \rbrace)} \lVert x- x^0 \rVert < + \infty.
\label{eq:assumption_par_front}
\end{align}
Then, there exists $R \ge 0$ with
\begin{align*}
    \sup_{z \in \H} \sigma_k(z) \le \frac{R}{(k+1)^2}.
\end{align*}
\label{thm:convergence_u_0}
\end{theorem}
\begin{proof}
Theorem \ref{thm:convergence_sigma_k_z} gives for all $z \in \H$
\begin{align*}
    \sigma_{k}(z) \le \frac{2 \left( \lVert x^1- z \rVert^2 + \lVert x^2 - z \rVert^2 \right)}{s (k+1)^2 }. 
\end{align*}
Taking a supremum over this inequality we get
\begin{align*}
    \sup_{f^* \in f(\mathcal{L^*})} \inf_{z \in f^{-1}(f^*)}\sigma_{k}(z) \le \sup_{f^* \in f(\mathcal{L^*})} \inf_{z \in f^{-1}(f^*)}\frac{2 \left( \lVert x^1- z \rVert^2 + \lVert x^2 - z \rVert^2 \right)}{s (k+1)^2}.
\end{align*}
Since $x^1, x^2 \in \mathcal{L}(f(x^0))$ assumption \eqref{eq:assumption_par_front} yields
\begin{align*}
    \sup_{f^* \in f(\mathcal{L^*})} \inf_{z \in f^{-1}(f^*)}\frac{2 \left( \lVert x^1- z \rVert^2 + \lVert x^2 - z \rVert^2 \right)}{s (k+1)^2} \le \frac{4R}{s(k+1)^2},
\end{align*}
with 
\begin{align}
    \label{eq:definition_R}
    R = \max_{j= 1,2} \left\lbrace \sup_{f^* \in f(\mathcal{L^*})} \inf_{z \in f^{-1}(f^*)} \lVert x^j - z \rVert^2 \right\rbrace.
\end{align}
It remains to show that
\begin{align*}
    \sup_{z \in \H} \sigma_k(z) = \sup_{f^* \in f(\mathcal{L^*})} \inf_{z \in f^{-1}(f^*)}\sigma_{k}(z).
\end{align*}
Writing out the definition of $\sigma_k(z)$, we get
\begin{align*}
    & \sup_{f^* \in f(\mathcal{L^*})} \inf_{z \in f^{-1}(f^*)}\sigma_{k}(z) 
    = \sup_{f^* \in f(\mathcal{L^*})} \inf_{z \in f^{-1}(f^*)} \min_{i =1, \dots, m}\left( f_i(x^k) - f_i(z) \right) \\
    = & \sup_{f^* \in f(\mathcal{L^*})} \min_{i =1, \dots, m} \left( f_i(x^k) - f_i^* \right) 
    =  \sup_{z \in \mathcal{L}^*} \min_{i =1, \dots, m} \left( f_i(x^k) - f_i(z) \right) \\
    = & \sup_{z \in \H} \min_{i =1, \dots, m} \left( f_i(x^k) - f_i(z) \right).
\end{align*}
\end{proof}
The function $u_0(x) = \sup_{z \in \H} \min_{i=1, \dots, m} f_i(x) - f_i(z)$ attains the value zero if and only if $x$ is weakly Pareto optimal. Theorem \ref{thm:convergence_u_0} shows that $u_0(x^k) = \mathcal{O}(k^{-2})$.
\subsection{Relation to Tanabe's Accelerated Multiobjective Gradient Method}
In the recent preprint \cite{Tanabe2022_2}, Tanabe, Fukuda and Yamashita define an accelerated proximal gradient method for MOPs with objective functions that have a separable structure of the form $f_i(x) = g_i(x) + h_i(x)$, where $g_i: \R^n \to \R$ is convex, continuously differentiable with $L$-Lipschitz continuous gradient and $h_i:\R^n \to \R$ is convex, lower semicontinuous and proper for all $i = 1,\dots, m$. Since we only treat the case of smooth objective functions $f_i$, we set from here on $h_i \equiv 0$. Tanabe et al.\ discovered their method using techniques different from the ones used throughout this paper, using the concept of merit functions. We will not recite their method here but refer the reader to \cite{Tanabe2022_2}. To understand the similarity between their method and Algorithm \ref{algo:ACC_GRAD}, we investigate the quadratic optimization problems that have to be solved in each iteration of the methods, respectively. In the method from \cite{Tanabe2022_2}, the step direction is computed by solving a quadratic optimization problem with the following objective function $\Psi:\R^m \to \R$,
\begin{align*}
    \Psi(\theta) \coloneqq \frac{s}{2}\lVert \sum_{i=1}^m \theta_i\nabla f_i(y^k) \rVert^2 + \sum_{i =1}^m \theta_i \left(f_i(x^k) - f_i(y^k)\right).
\end{align*}
Using the first order approximation $f_i(y^k) - f_i(x^k) \approx \langle \nabla f_i(y^k), y^k - x^k \rangle$, we get
\begin{align*}
     \Psi(\theta) \approx \frac{s}{2}\lVert \sum_{i=1}^m \theta_i\nabla f_i(y^k) \rVert^2 + \langle \sum_{i =1}^m \theta_i \nabla f_i(y^k), x^k -y^k \rangle.
\end{align*}
Minimizing $\Psi(\theta)$ is equivalent to minimizing the function $\Phi:\R^m \to \R$,
\begin{align*}
    \Phi(\theta) \coloneqq & \frac{s^2}{2}\lVert \sum_{i=1}^m \theta_i\nabla f_i(y^k) \rVert^2 + \langle s\sum_{i =1}^m \theta_i \nabla f_i(y^k), x^k -y^k \rangle + \frac{1}{2} \lVert x^k - y^k \rVert^2\\
    = &\frac{1}{2}\lVert s \sum_{i=1}^m \theta_i\nabla f_i(y^k) + (x^k - y^k) \rVert^2.
\end{align*}
Using $x^k - y^k = -\frac{k-1}{k+2}(x^k - x^{k-1})$ we note that $\Phi(\theta)$ is in fact the objective function of the quadratic optimization problem \eqref{eq:QOP_a_3}. After this observation, it is not surprising that the method by Tanabe et al.\ shows covergence behavior similar to Algorithm \ref{algo:ACC_GRAD}.

\section{Improving the Numerical Efficiency}
\label{sec:improving numerical efficiency}
First order methods for multiobjective optimization that are based on the steepest descent method by Fliege and Svaiter \cite{Fliege2000} require the solution of a quadratic subproblem in each iteration. Computing the solutions of these problems is computational demanding. In the following subsection, we present a possible approach to overcome this problem.
\subsection{A Multiobjective Gradient Method without Quadratic Subproblems}
In this subsection, we define a method based on Algorithm \ref{algo:ACC_GRAD} which does not require the solution of a quadratic subproblem in each iteration. In Subsection \ref{subsec:intro_new_acc_imog'}, we derived Algorithm \ref{algo:ACC_GRAD} from the scheme
\begin{align}
    \frac{3}{k} (x^{k+1} - x^k) + \proj_{sC(y^k) + (x^{k+1} - 2x^k + x^{k-1})}(0) = 0,
\end{align}
which can be interpreted as a discretization of the differential equation
\begin{align*}
    \frac{3}{t}\dot{x}(t) + \proj_{C(x(t)) + \ddot{x}(t)}(0) = 0.
\end{align*}
If, instead, we use the discretization
\begin{align}
    \frac{3}{k} (x^{k} - x^{k-1}) + \proj_{sC(y^k) + (x^{k+1} - 2x^k + x^{k-1})}(0) = 0,
\label{eq:finite_diff_scheme_a_3_alt}
\end{align}
we obtain a different method. Lemma \ref{lem:proj_1} gives a formula to compute $x^{k+1}$
\begin{align}
    x^{k+1} & = -\frac{3}{k} (x^k - x^{k-1}) - s\sum_{i=1}^m \theta_i^k \nabla f_i(x^k) + 2x^k - x^{k-1}, \\
    & = x^k + \frac{k - 3}{k} (x^k - x^{k-1}) - s\sum_{i = 1}^m \theta_i^k \nabla f_i(x^k),
\label{eq:greedy_scheme}
\end{align}
where $\theta^k \in \R^m$ is a solution to the problem
\begin{align}
\label{eq:lin_problem2}
\begin{split}
    \min -  \sum_{i = 1}^m \theta_i \langle\nabla f_i(x^k), x^k - x^{k-1} \rangle\,\, \st \,\,\theta \ge 0 \,\,\text{and}\,\, \sum_{i = 1}^m \theta_i = 1.
\end{split}
\end{align}
This can be solved efficiently by computing $m$ inner products. After changing $\frac{k - 3}{k}$ into $\frac{k-1}{k + 2}$ in \eqref{eq:greedy_scheme}, we define Algorithm \ref{algo:ACC_GRAD_wo_Q}.
\begin{algorithm} 
	\caption{Accelerated multiobjective gradient method without quadratic subproblems}
	\label{algo:ACC_GRAD_wo_Q}
	\begin{algorithmic}[1] 
		\Require Choose $x^0 = x^1 \in \H$, $s > 0$ and set k = 1.
		\State Set $y^k = x^k + \frac{k-1}{k+2}(x^k - x^{k-1})$.
		\State Compute $j =\argmax_{i = 1, \dots, m}   \langle \nabla f_i(y^k), x^k - x^{k-1} \rangle$.
		\State Set $x^{k+1} = y^k - s\nabla f_j(y^k)$
		\If{stopping condition is true}
            \State Stop.
        \Else 
            \State Update $k \leftarrow k+1$ and go to step 1.
        \EndIf
	\end{algorithmic} 
\end{algorithm}
There is no proof of convergence for the method defined by Algorithm \ref{algo:ACC_GRAD_wo_Q} but we discuss its numerical behavior in Section \ref{sec:numerical_experiments}.
Also, convergence can be guaranteed by switching from the significantly faster Algorithm \ref{algo:ACC_GRAD_wo_Q} to Algorithm \ref{algo:ACC_GRAD} as soon as some heuristic criterion is met.

\subsection{Backtracking for unknown Lipschitz Constants}
In all presented algorithms, we can include backtracking if the Lipschitz constants of the gradients $\nabla f_i$ of the objective functions are unknown. We can do this as stated in \cite{Beck2009, Tanabe2022_2}. To include backtracking, we choose an initial step size $s_0 > 0$ and a parameter $\sigma \in (0,1)$. In all discussed algorithms there is a step $x^{k+1} = w^k - s d^k$, with $d^k \in \H$ and $w^k = x^k$ or $w^k = y^k$. One can replace this step with $x^{k+1} = w^k - s_k d^k$, with a step size $s_k$ that is determined using backtracking. We choose in every step $s_k = \sigma^{l_k} s_{k-1}$ where $l_k \ge 0$ is the smallest nonnegative integer satisfying for all $i = 1,\dots, m$
\begin{align*}
    f_i(w^k - \sigma^{l_k} s_{k-1} d^k) \le f_i(w^k) - \sigma^{l_k} s_{k-1} \langle \nabla f_i(w^k), d^k \rangle + \frac{\sigma^{l_k} s_{k-1}}{2} \lVert d^k \rVert^2.
\end{align*}
The sequence $(s_k)_{k \ge 0}$ is monotonically decreasing by definition. Under the condition that the objective functions posses $L$-Lipschitz continuous gradients, it is guaranteed that the sequence $(s_k)_{k \ge 0}$ is constant from same $k$ on. This is true since $s_k$ can only decrease as long as $s_k > \frac{1}{L}$. Therefore, $s_k$ can only decrease finitely many times until it reaches a point where $s_k \le \frac{1}{L}$. Using this observation, we can include backtracking in Algorithm \ref{algo:ACC_GRAD} and still use the proofs of Theorem \ref{thm:convergence_sigma_k_z} and Theorem \ref{thm:convergence_u_0} to show that the same convergence results can be achieved.
\section{Numerical Examples}
\label{sec:numerical_experiments}
In this section, we present the typical behavior of our algorithms on two test problems. We compare Algorithms \ref{algo:ACC_GRAD} and \ref{algo:ACC_GRAD_wo_Q} with the steepest descent method by Fliege and Svaiter with constant step sizes \cite{Fliege2000}. Throughout this section we denote the steepest descent method by SD, Algorithm \ref{algo:ACC_GRAD} by AccG (accelerated gradient method) and Algorithm \ref{algo:ACC_GRAD_wo_Q} by AccG w\textbackslash o Q (accelerated gradient method without quadratic subproblems). We implemented all codes using Matlab R2021b and executed the algorithms on a machine with a 2.80 GHz Intel Core i7 processor and 48 GB memory. We solved the quadratic subproblems for SD and AccG using the built-in Matlab function quadprog.  
\subsection{Example 1: A Convex MOP with three Objective Functions}
\label{subsec:Example1}
\begin{figure}
\centering
    \begin{subfigure}[b]{0.46\textwidth}
        \centering
        \includegraphics[width=\linewidth]{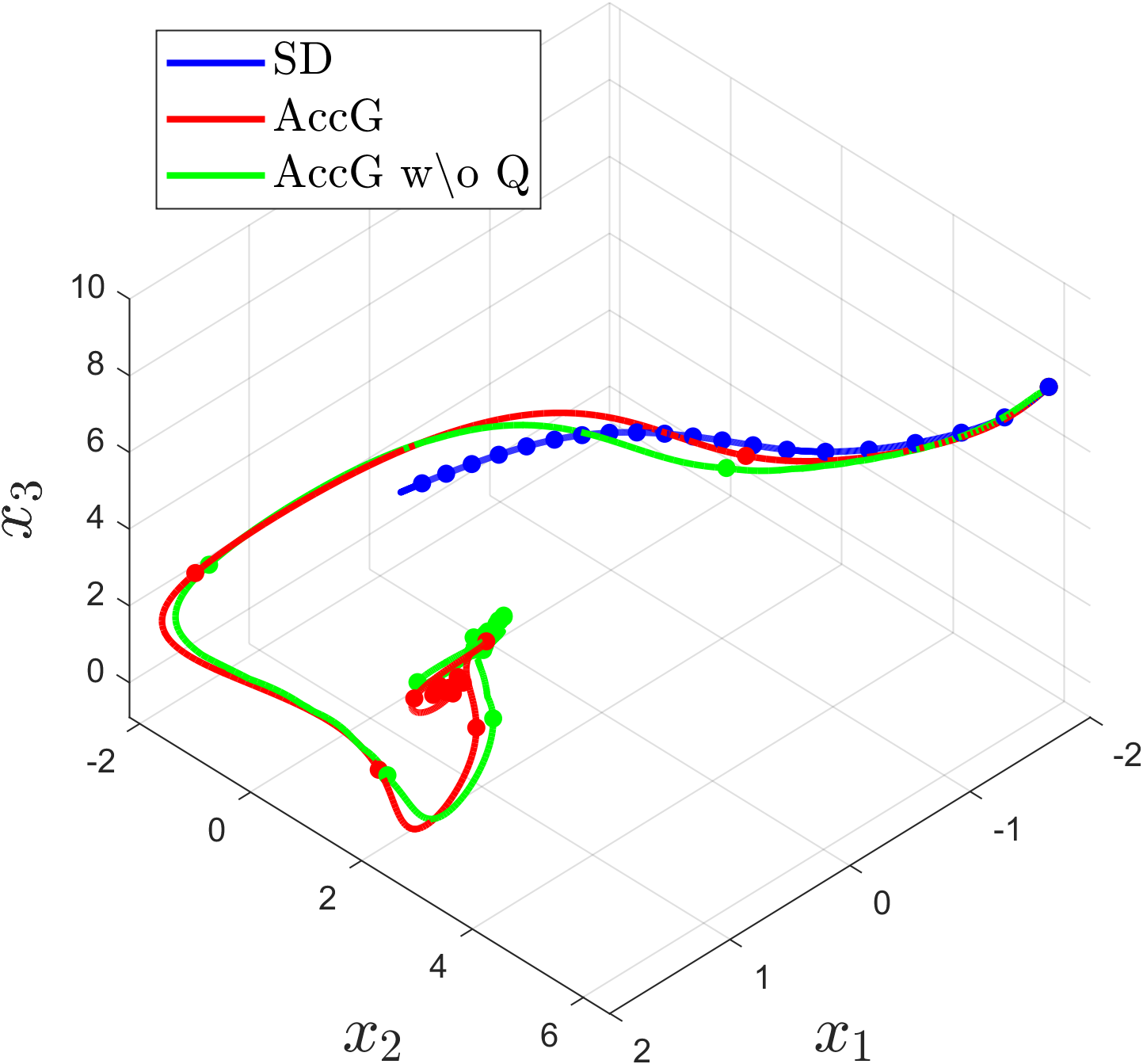}
  \caption{}\label{fig:iteration_sequence_a}
    \end{subfigure}
    \hfill
    \begin{subfigure}[b]{0.42\textwidth}
        \centering
        \includegraphics[width=\linewidth]{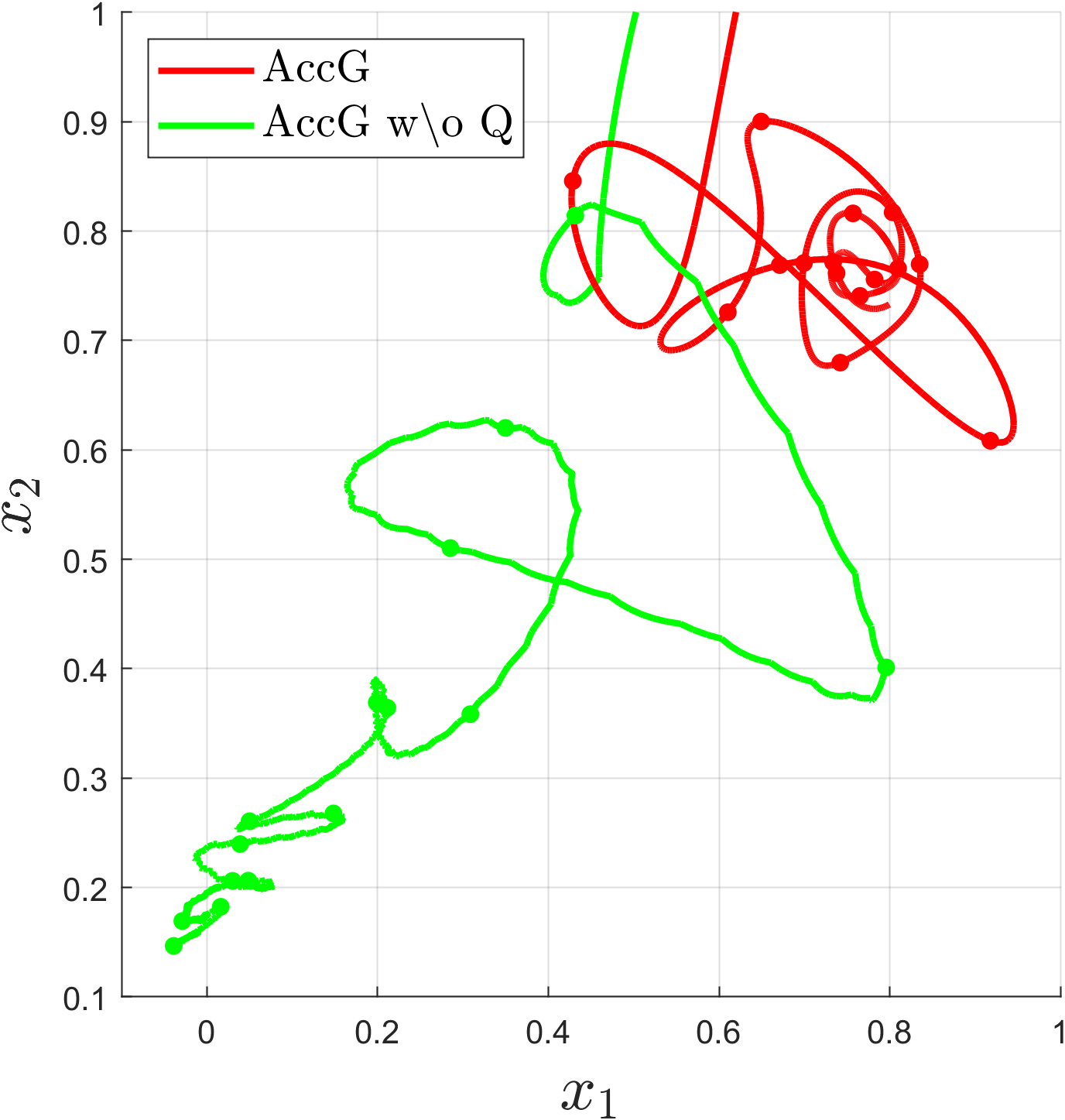}
  \caption{}\label{fig:iteration_sequence_b}
    \end{subfigure}
        \caption{Coordinates $(x_1, x_2, x_3)$ of the sequences $(x^k)_{k \ge 0}$ for SD, AccG and AccG w\textbackslash o Q. Line plot for $1000$ iterations with a filled circle every $50$ iterations to compare the velocities.}
        \label{fig:iteration_sequence}
        \vspace{4mm}
     \begin{subfigure}[b]{0.3\textwidth}
        \centering
        \includegraphics[width=\linewidth]{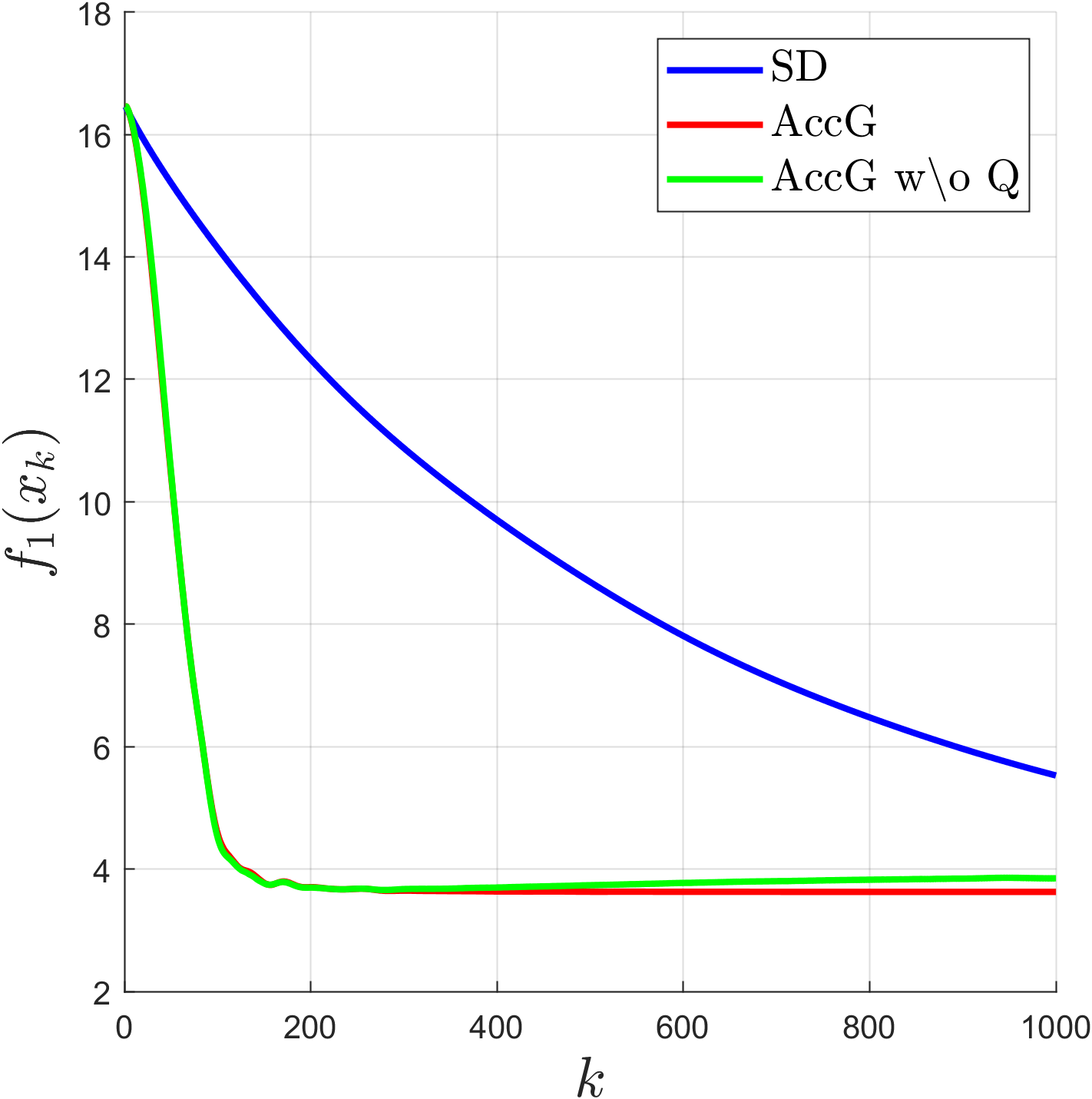}
        \caption{Function values $f_1(x^k)$}\label{fig:function_values_a}
    \end{subfigure}
    \hfill
    \begin{subfigure}[b]{0.3\textwidth}
        \centering
        \includegraphics[width=\linewidth]{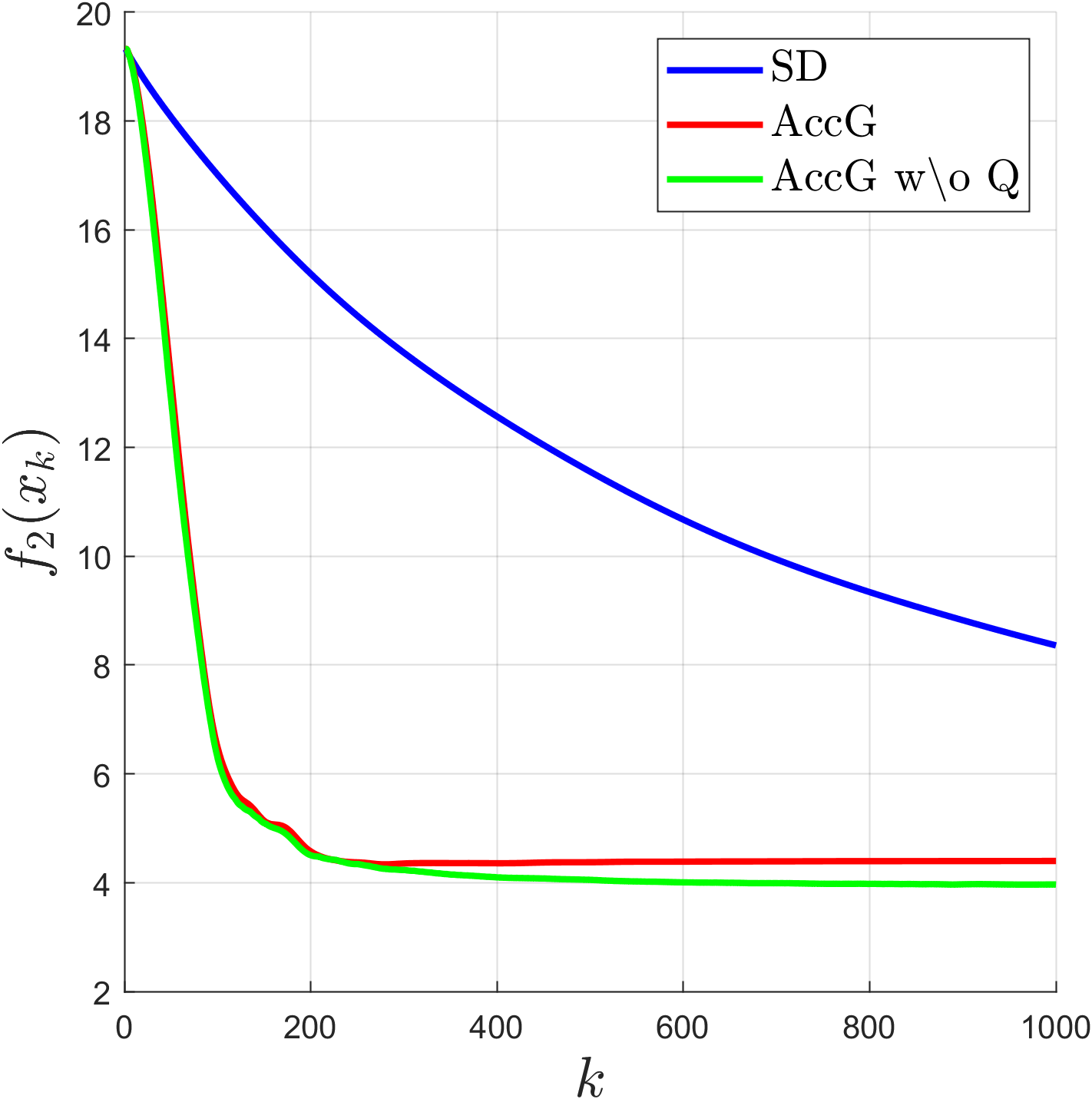}
        \caption{Function values $f_2(x^k)$}\label{fig:function_values_b}
     \end{subfigure}
     \hfill
     \begin{subfigure}[b]{0.3\textwidth}
         \centering
         \includegraphics[width=\linewidth]{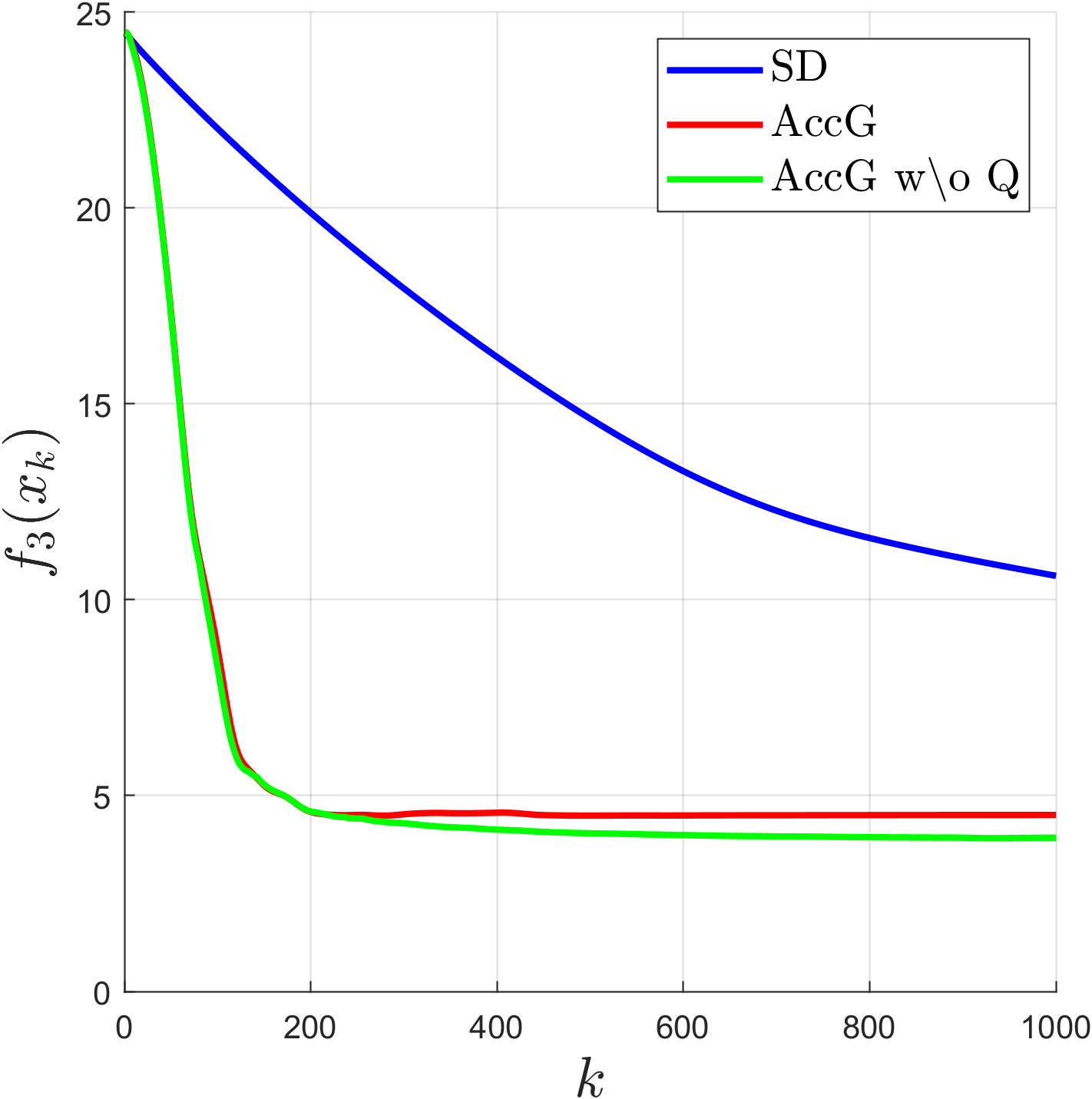}
  \caption{Function values $f_3(x^k)$}\label{fig:function_values_c}
     \end{subfigure}
        \caption{Function values $(f_i(x^k))_{k\ge 0}$ of the iterates for the objective functions $i = 1,2,3$ for the different algorithms.}
        \label{fig:function_values}
\end{figure}
\begin{figure}
\centering
    \begin{subfigure}[b]{0.4\textwidth}
        \centering
        \includegraphics[width=\linewidth]{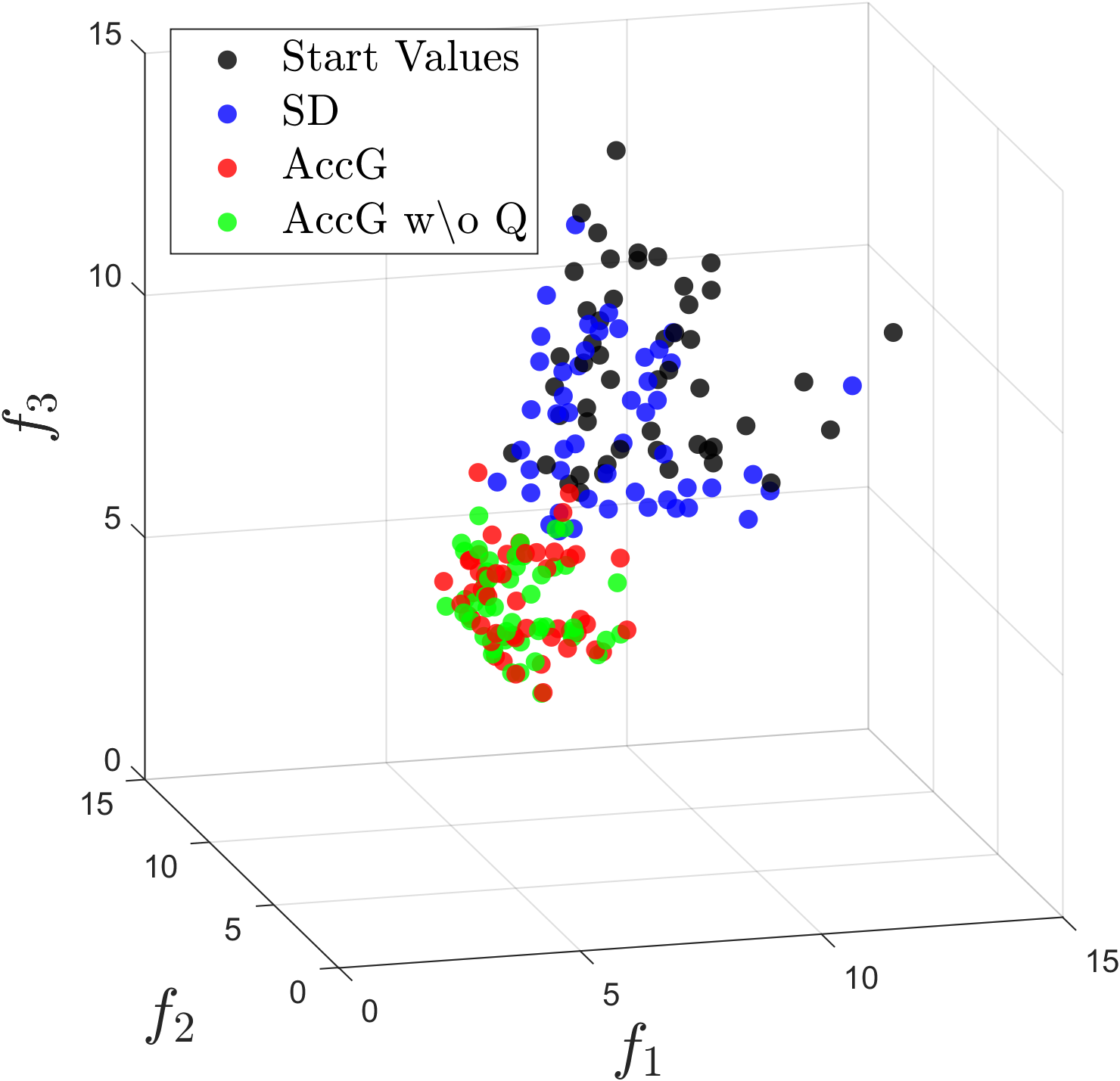}
  \caption{$k_{\max} = 50$}\label{fig:pareto_front_a}
    \end{subfigure}
    \hfill
    \begin{subfigure}[b]{0.4\textwidth}
        \centering
        \includegraphics[width=\linewidth]{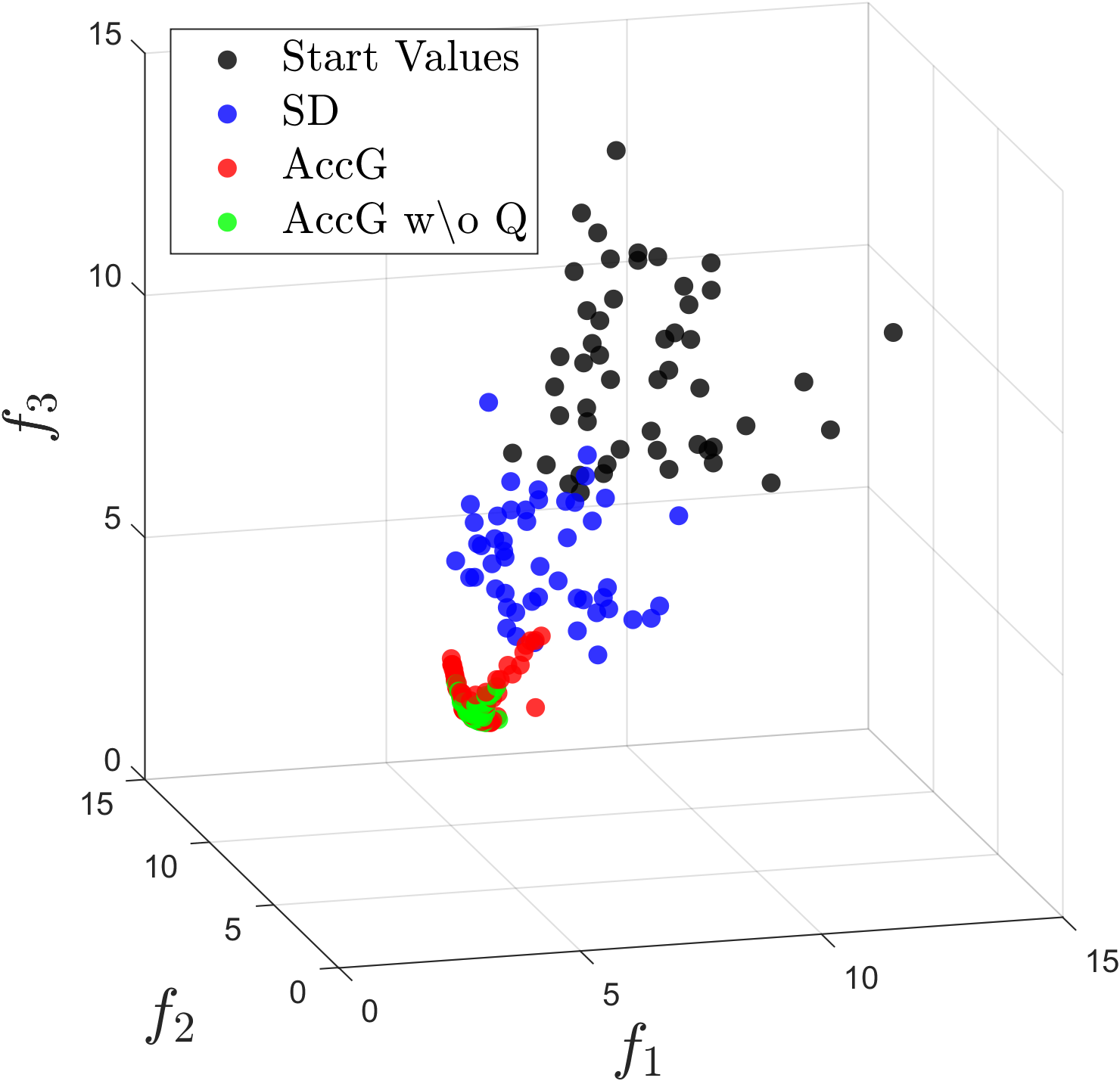}
  \caption{$k_{\max} = 250$}\label{fig:pareto_front_b}
    \end{subfigure}
     \begin{subfigure}[b]{0.4\textwidth}
        \centering
        \includegraphics[width=\linewidth]{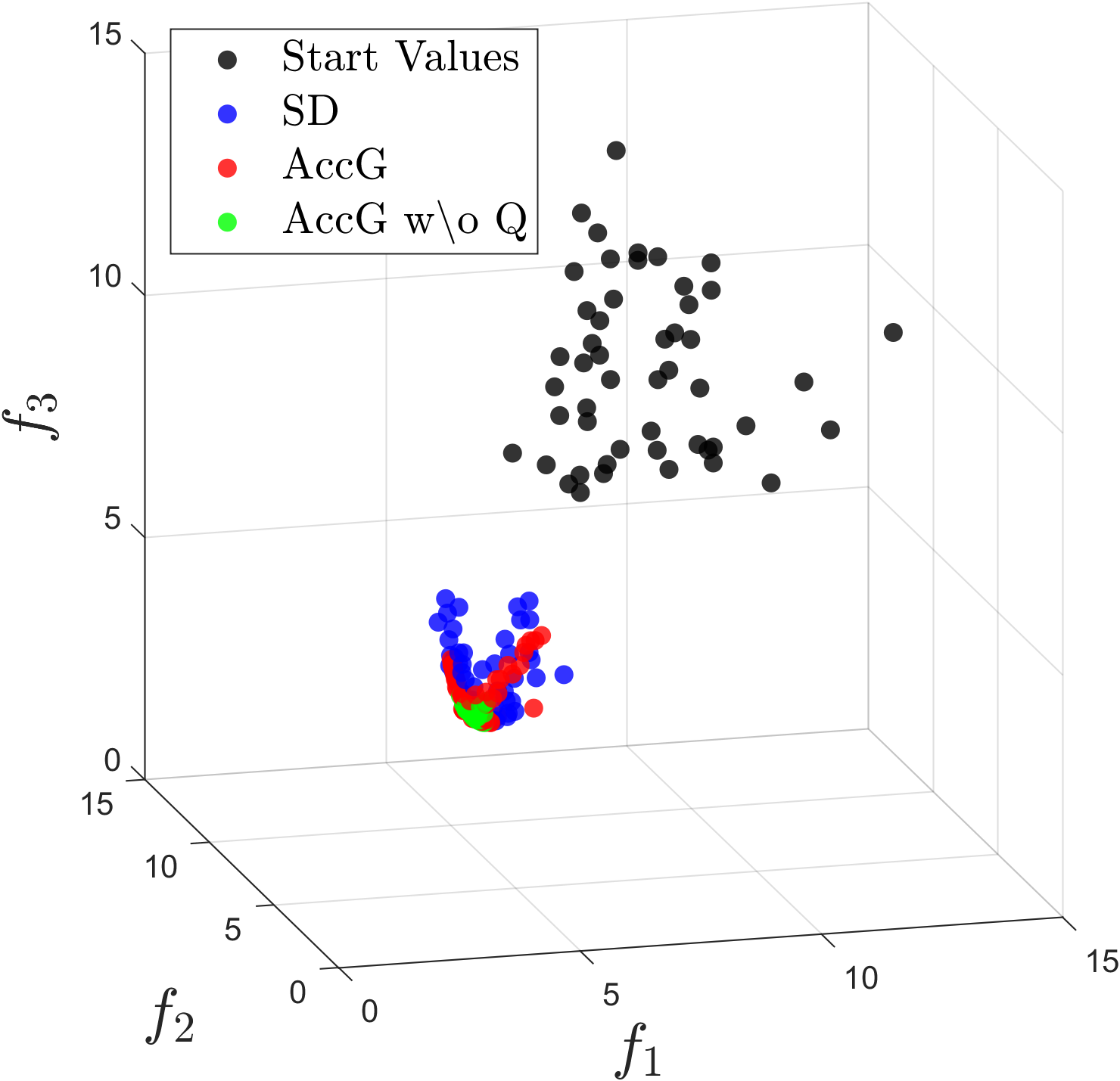}
        \caption{$k_{\max} = 1000$}\label{fig:pareto_front_c}
    \end{subfigure}
    \hfill
    \begin{subfigure}[b]{0.4\textwidth}
        \centering
        \includegraphics[width=\linewidth]{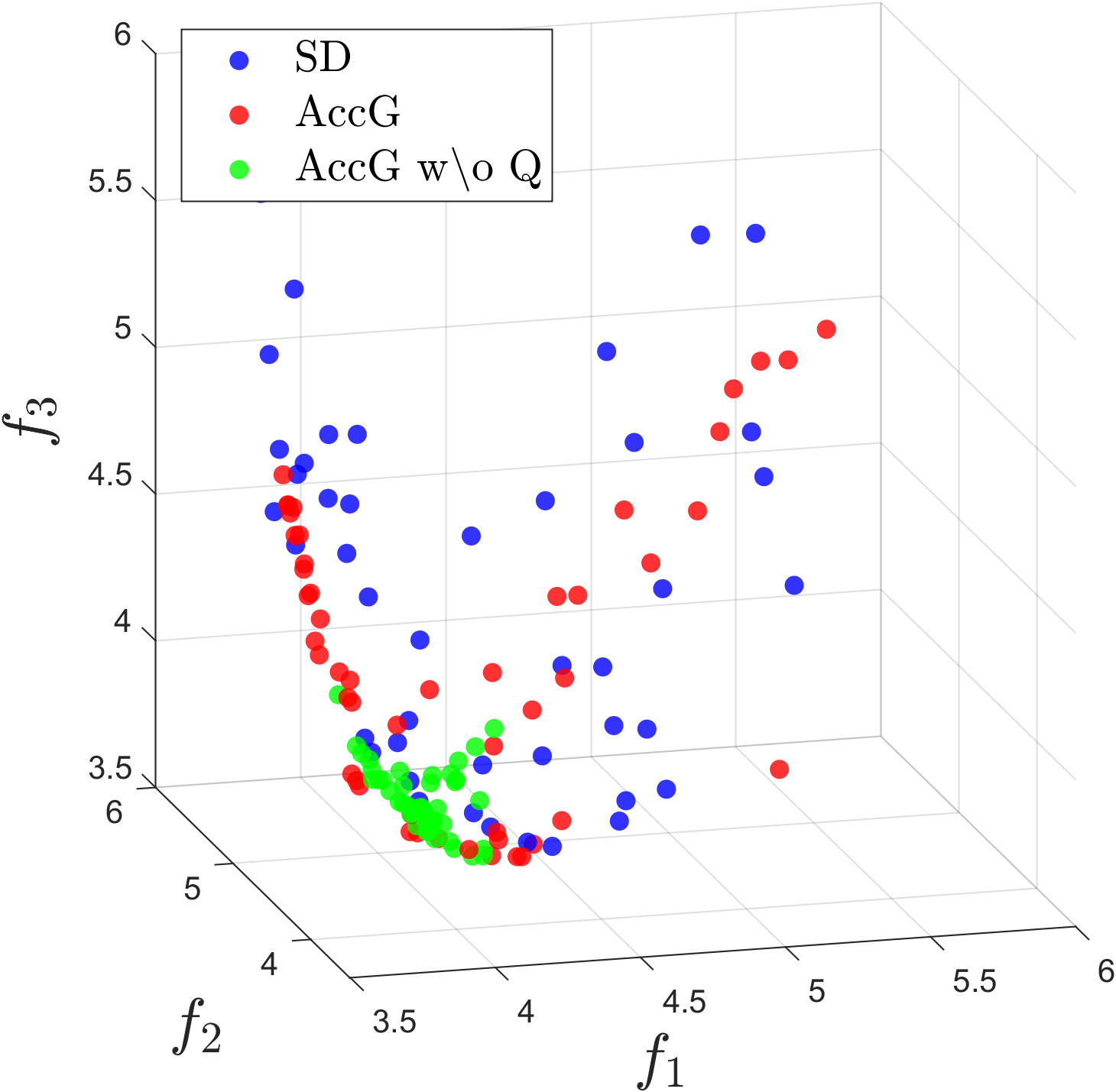}
        \caption{$k_{\max} = 1000$}\label{fig:pareto_front_d}
    \end{subfigure}
    \caption{Function values of the objective functions in the image space for the different algorithms and a different maximum number of iterations $k_{\max} = 50, 250, 1000$.}
    \label{fig:pareto_front}
\end{figure}
In our first example, we choose a problem with input dimension $n = 20$ and three objective functions ($m = 3$). We define the objective functions using the following parameters. For $p = 50$ and $i = 1,2,3$ we generate matrices $A^i = \left( a_1^i, \dots, a_p^i \right)^T  \in \R^{p \times n}$ with $a_j^i \in \R^n$ for $j = 1,\dots, p$ and vectors $b^i \in \R^p$. Then, for $i = 1,2,3$, we define the objective functions
\begin{align*}
    f_i :\R^n \to \R, \quad x \mapsto \ln\left( \sum_{j = 1}^p \exp\left( (a_j^i)^T x \right) \right).
\end{align*}
For the first experiment we randomly generate Matrices $A^i \in \R^{p \times n}$ and vectors $b_i \in \R^p$ with entries uniformly sampled in $[-1, 1]$ for $i = 1,2,3$. The starting vector $x_0$ is uniformly randomly drawn from $[-15, 15]^n$. We use the step size $s = \expnumber{5}{-2}$ and execute maximally $k_{\max} = 1000$ iterations. Figure \ref{fig:iteration_sequence} contains plots of the sequences $(x^k)_{k \ge 0}$ for the different algorithms. In Figure \ref{fig:iteration_sequence_a}, one sees that the sequences generated with AccG and AccG w\textbackslash o Q advance much faster in the beginning, while the velocity for the sequence generated with SD remains constant. The sequences generated by AccG and AccG w\textbackslash o Q give very similar trajectories in the beginning. This result is intuitive given that the schemes in the algorithms are derived from different discretizations of the same differential equation. However, this result is still surprising keeping in mind that in Algorithm \ref{algo:ACC_GRAD_wo_Q} we do not solve a quadratic subproblem in each iteration. Only in Figure \ref{fig:iteration_sequence_b}, we see that the sequences differ more substantially in the long run. It is also noteworthy that the sequence generated by AccG is smoother compared to the trajectory generated by AccG w\textbackslash o Q. This is due to the fact that in AccG w\textbackslash o Q we choose one of the gradients of the objective functions for the gradient component of the step direction while in AccG we choose an element of the convex hull of the gradients. AccG and AccG w\textbackslash o Q are superior to SD in terms of convergence of the function values for all objective functions, as shown in Figure \ref{fig:function_values}. AccG and AccG w\textbackslash o Q experience fast convergence within the first $200$ iterations. Comparing the different objective functions in Figures \ref{fig:function_values_a}, \ref{fig:function_values_b} and \ref{fig:function_values_c}, we see that AccG and AccG w\textbackslash o Q yield outputs with similar function values for all objective functions.

In a second experiment, we execute all algorithms for $50$ starting values uniformly sampled in $[-5,5]^n$ with step size $s = \expnumber{5}{-2}$. We use the stopping criterion $\lVert f(x^k) - f(x^{k-1}) \rVert_{\infty} < \expnumber{1}{-4}$ to stop the algorithms if the function values do not change significantly. In Figure \ref{fig:pareto_front_a}, we perform up to $k_{\max} = 50$, in Figure \ref{fig:pareto_front_b} up to $k_{\max} = 250$ and in Figures \ref{fig:pareto_front_c} and \ref{fig:pareto_front_d} up to $k_{\max} = 1000$ iterations. Similar to the results observed in Figures \ref{fig:iteration_sequence} and \ref{fig:function_values}, AccG and AccG w\textbackslash o Q advance much faster in the beginning compared to SD. Comparing Figures \ref{fig:function_values_b} and \ref{fig:function_values_c}, we see that after $250$ iterations the function values for the accelerated methods are converging or the stopping conditions were met. The different behavior of the accelerated methods can be observed in Figure \ref{fig:pareto_front_d}. While the solutions of AccG are farther spread, it looks like the solutions of AccG w\textbackslash o Q are drawn towards the center of the Pareto front. Altogether, the accelerated methods perform better for this problem in terms of convergence speed of the function values. In Table \ref{table:iterations_ex1} the total number of iterations and computation times for the experiments are listed. The accelerated methods require fewer iterations. Compared to SD, AccG requires only approximately 25 \% and AccG w\textbackslash o Q only approximately 50 \% iterations. For the computation times the results are different. SD and AccG behave similar, with AccG requiring approximately 25 \% of the computation time that is required for SD. However, AccG w\textbackslash o Q needs less the 2 \% of the time which is consumed by AccG. This improvement stems from the quadratic optimization problems that not required in AccG w\textbackslash o Q.
\begin{table}[ht!]
\centering
\begin{tabular}{c c c c}
 \vspace{1mm}
  & SD & AccG & AccG w\textbackslash o Q\\
 \hline
 total iterations & 49924 & 12230 & 25906 \\ 
 total time & $436.54 \,\text{s}$ & $100.94 \,\text{s}$ & $1.61 \,\text{s}$ \\
 \hline
\end{tabular}
\caption{Total iterations and computation times for algorithm executions using parameters $s = \expnumber{5}{-2}$, $k_{\max} = 1000$ and stopping condition $\lVert f(x^k) - f(x^{k-1}) \rVert_{\infty} < \expnumber{1}{-4}$ for $50$ start values uniformly sampled in $[-15, 15]^n$.}
\label{table:iterations_ex1}
\end{table}

\subsection{Example 2: A Nonconvex MOP with two Objective Functions}
For our second test problem, we choose an example from \cite{Witting2012} with input dimension $n = 2$ and the two objective functions
\begin{align*}
    f_1(x) = \frac{1}{2}(\sqrt{1 + (x_1 + x_2)^2} + \sqrt{1 + (x_1 - x_2)^2} + x_1 - x_2) + \lambda \exp{(-(x_1-x_2)^2 )},\\
    f_2(x) = \frac{1}{2}(\sqrt{1 + (x_1 + x_2)^2} + \sqrt{1 + (x_1 - x_2)^2} - x_1 + x_2) + \lambda \exp{(-(x_1-x_2)^2)},
\end{align*}
with $\lambda = 0.6$. For the multiobjective optimization problem \eqref{eq:MOP} with these objective functions it can easily be verified that the Pareto set is
\begin{align*}
    P = \left\lbrace x \in \R^2 \,\,:\,\, x_1 + x_2 = 0 \right\rbrace.
\end{align*}
In the first experiment, we execute Algorithms SD, AccG and AccG w\textbackslash o Q with the starting vector $x^0 = (1,2)^T$ and perform $k_{\max} = 1000$ iterations. The step size is set to $s = \expnumber{5}{-3}$. The sequences and function values of the objective functions are shown in Figure \ref{fig:sequence_n_f_val_ex2}. Similarly to the first experiment in Figure \ref{fig:sequence_ex2}, the sequences $(x^k)_{k \ge 0}$ of the  accelerated methods advance faster in the beginning. While Algorithms SD and AccG converge to the same element in the Pareto set the algorithm AccG w\textbackslash o Q produces a trajectory that deviates from the trajectories of SD and AccG and moves to a different part of the Pareto set. The values of the objective functions in Figures \ref{fig:f1_val_ex2} and \ref{fig:f2_val_ex2} indicate a similar behavior. For the accelerated methods we have faster decrease in the beginning and we note that the function values for SD and AccG converge to similar values. In Figure \ref{fig:pareto_front_ex2} we use $100$ random starting points uniformly sampled in $[-2,2]^2$. For the experiments we use different maximal numbers of iterations $k_{\max}$. In addition we stop the algorithm if $\lVert f(x^k) - f(x^{k-1}) \rVert_{\infty} < \expnumber{1}{-4}$. Comparing Figures \ref{fig:pareto_front_a_ex2}, \ref{fig:pareto_front_b_ex2} and \ref{fig:pareto_front_c_ex2}, we note that the objective function values of the accelerated methods decrease much faster in the beginning. For $k_{\max} = 100$ Algorithms AccG and AccG w\textbackslash o Q yield solutions that are distributed along the Pareto front. In Table \ref{table:iterations_ex2} we list the total number of iterations and total computation times for executions with up to $k_{\max} = 1000$ iterations with stopping condition $\lVert f(x^k) - f(x^{k-1}) \rVert_{\infty} < \expnumber{1}{-4}$. Compared to SD, AccG needs only approximately 15 \% and AccG w\textbackslash o Q only approximately 51 \% iterations.
\begin{figure}
\centering
    \begin{subfigure}[b]{0.32\textwidth}
        \centering
        \includegraphics[width=\linewidth]{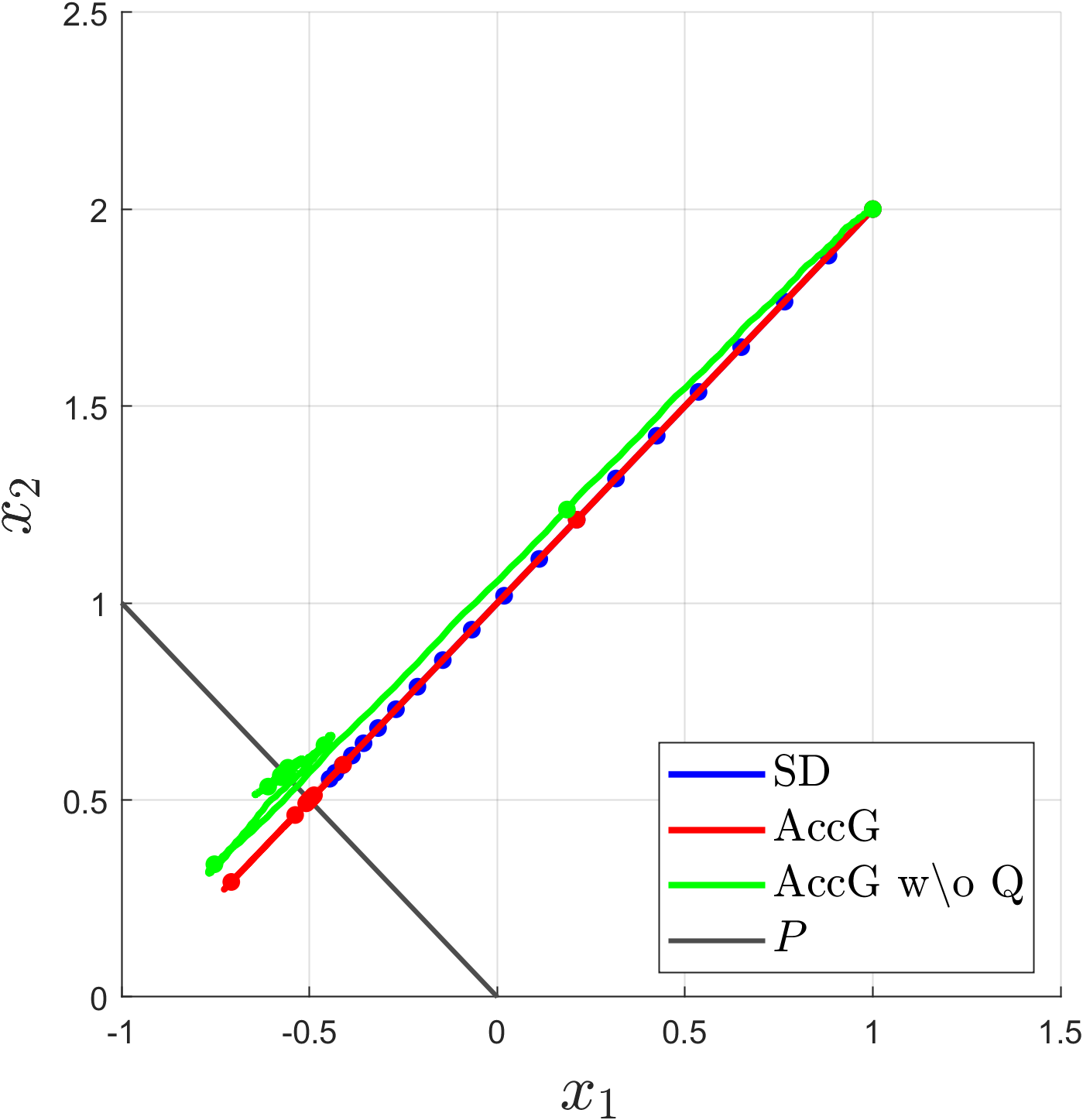}
  \caption{Sequences $(x^k)_{k \ge 0}$}\label{fig:sequence_ex2}
    \end{subfigure}
    \hfill
    \begin{subfigure}[b]{0.32\textwidth}
        \centering
        \includegraphics[width=\linewidth]{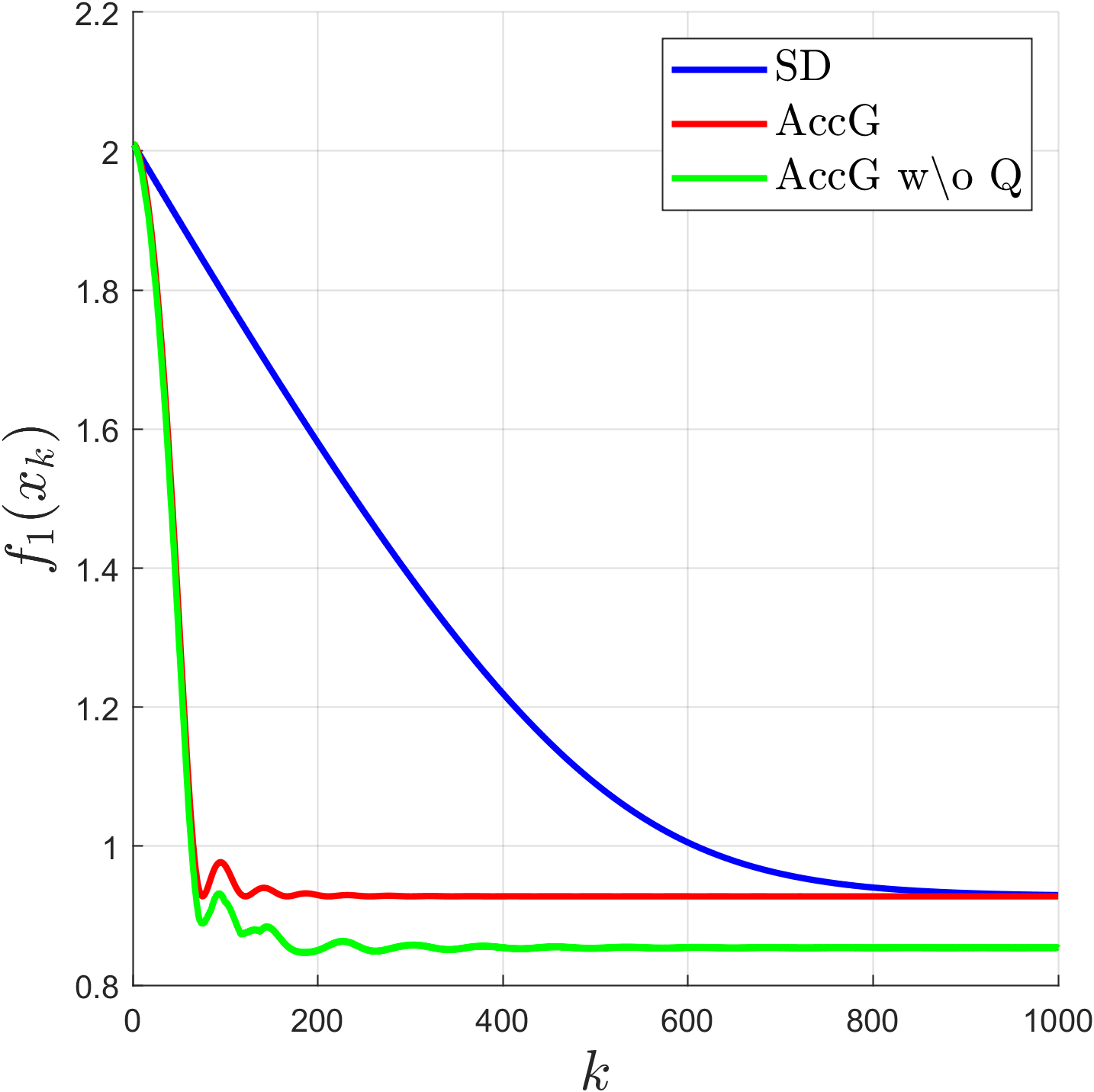}
  \caption{Function values $f_1(x^k)$}\label{fig:f1_val_ex2}
    \end{subfigure}
     \begin{subfigure}[b]{0.32\textwidth}
        \centering
        \includegraphics[width=\linewidth]{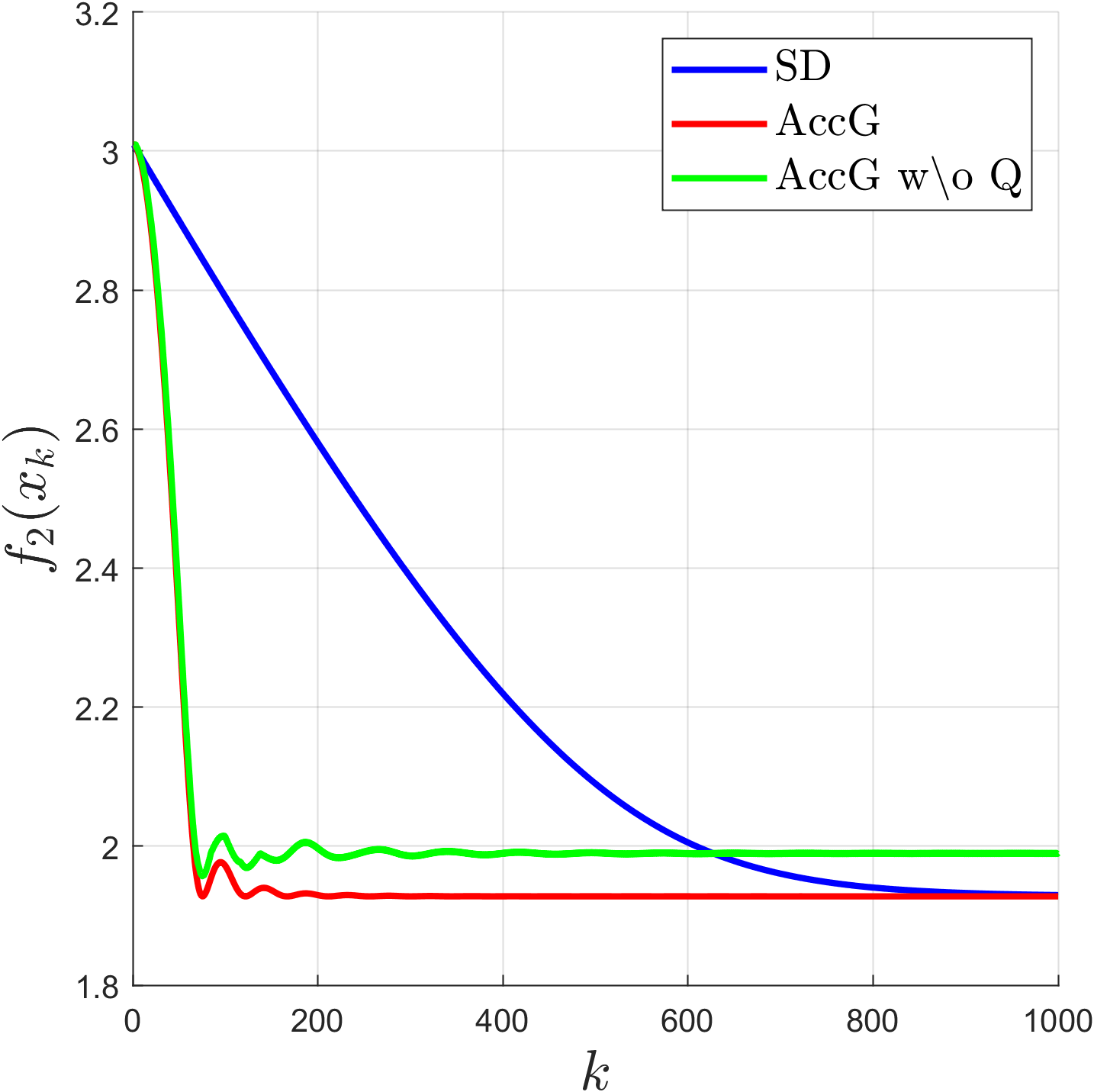}
        \caption{Function values $f_2(x^k)$}\label{fig:f2_val_ex2}
    \end{subfigure}
    \caption{Sequences $(x^k)_{k\ge 0}$ and function values $(f_i(x^k))_{k \ge 0}$ of iterates for $i = 1,2$. For the sequences we use a line plot for 1000 iterations with a filled circle every $50$ iterations to compare velocities.}
    \label{fig:sequence_n_f_val_ex2}
\end{figure}
\begin{figure}
\centering
    \begin{subfigure}[b]{0.32\textwidth}
        \centering
        \includegraphics[width=\linewidth]{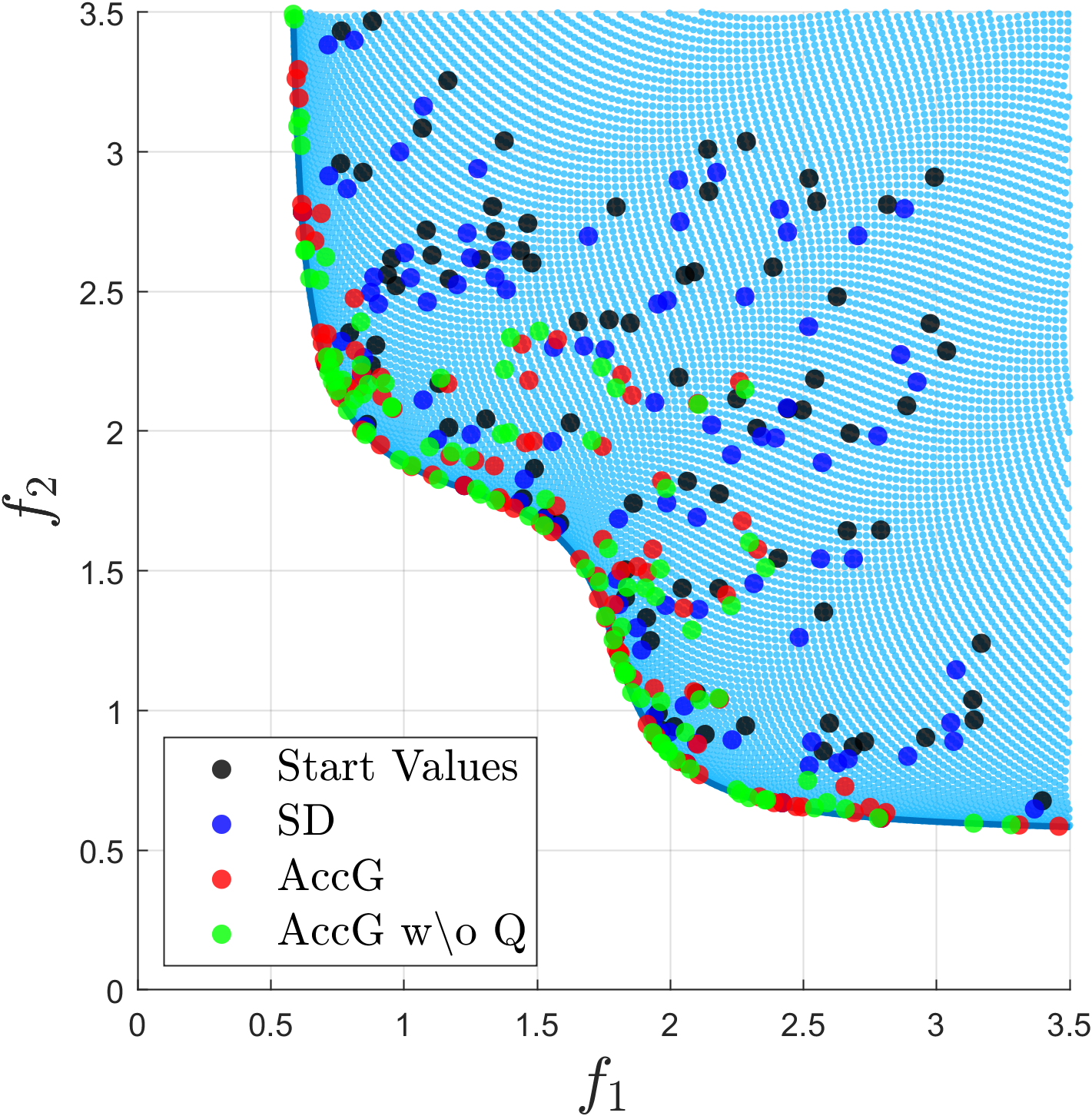}
  \caption{$k_{\max} = 50$}\label{fig:pareto_front_a_ex2}
    \end{subfigure}
    \hfill
    \begin{subfigure}[b]{0.32\textwidth}
        \centering
        \includegraphics[width=\linewidth]{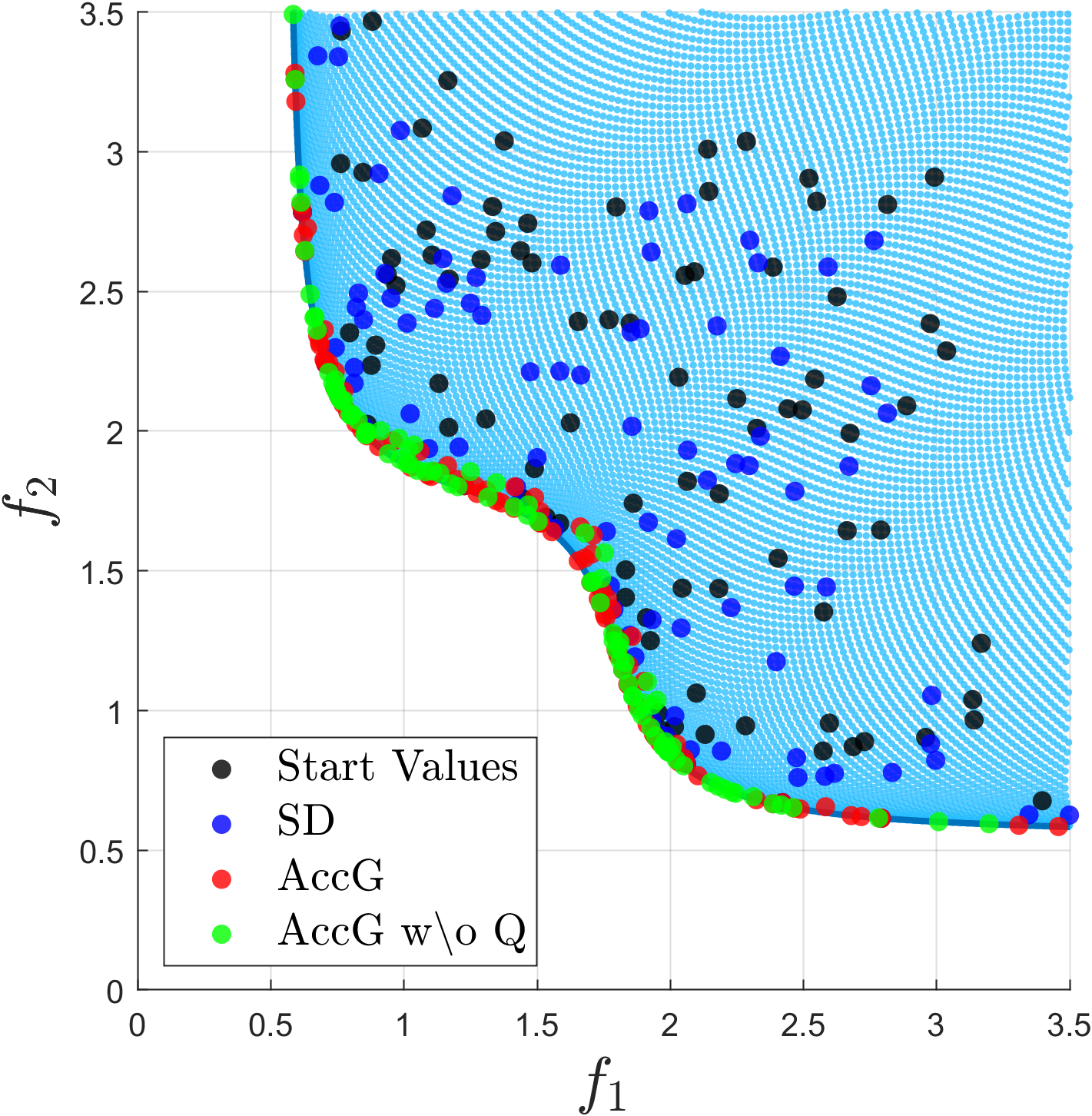}
  \caption{$k_{\max} = 100$}\label{fig:pareto_front_b_ex2}
    \end{subfigure}
     \begin{subfigure}[b]{0.32\textwidth}
        \centering
        \includegraphics[width=\linewidth]{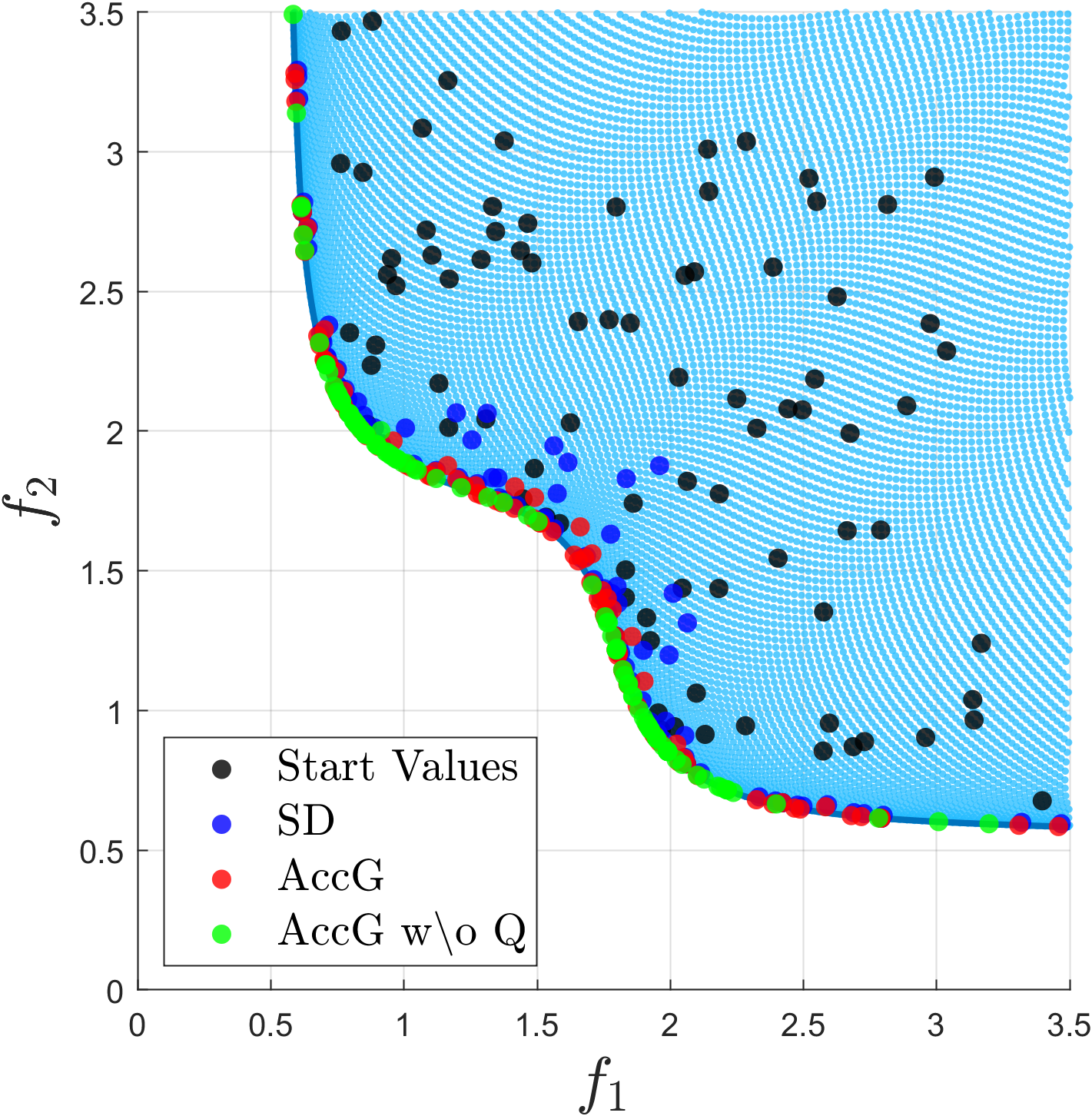}
        \caption{$k_{\max} = 500$}\label{fig:pareto_front_c_ex2}
    \end{subfigure}
    \caption{Values of the objective functions in the image space for different maximum numbers of iterations $k_{\max} = 50, 250, 1000$ for the different algorithms SD, AccG and AccG w\textbackslash o Q}
        \label{fig:pareto_front_ex2}
\end{figure}
\begin{table}[ht!]
\centering
\begin{tabular}{c c c c}
 \vspace{1mm}
  & SD & AccG &AccG w\textbackslash o Q\\
 \hline
 total iterations & 45543 & 6632 & 23034 \\
 total time & $431.75 \,\text{s}$ & $62.90 \,\text{s}$ & $0.31 \,\text{s}$ \\
 \hline
\end{tabular}
\caption{Total iterations and computation times for algorithm executions using parameters $h = \expnumber{5}{-3}$, $k_{\max} = 1000$ and stopping condition $\lVert f(x^k) - f(x^{k-1}) \rVert_{\infty} < \expnumber{1}{-4}$ for $100$ start values uniformly sampled in $[-2, 2]^n$.}
\label{table:iterations_ex2}
\end{table}
In another experiment we compare how the choice of the step size $s$ affects the solutions of AccG w\textbackslash o Q. We use the step sizes $s = \expnumber{5}{-3}, \expnumber{1}{-2}, \expnumber{5}{-2}$. For all executions we perform $k_{\max} = 1000$ iterations, with the stopping criterion $\lVert f(x^k) - f(x^{k-1}) \rVert_{\infty} < \expnumber{1}{-4}$. Comparing Figures \ref{fig:ParetoFront_1000_h_5e-3}, \ref{fig:ParetoFront_1000_h_1e-2} and \ref{fig:ParetoFront_1000_h_5e-2} we see that for the smallest step size $s = \expnumber{5}{-3}$ solutions are distributed on the whole Pareto front. For the biggest step size $s = \expnumber{5}{-2}$ Algorithm AccG w\textbackslash o Q yields solutions that cluster at two points of the Pareto front, which is not desirable in general. The two points where the solutions cluster correspond to the knee points in the Pareto front. This is not surprising since these points correspond to solutions where the individual gradients are similar in magnitude, which is why the solution is zig-zagging back and forth between the objectives in these locations.
However, this disadvantage is compensated by the fact that we do not need to solve a quadratic subproblem in every step. We need potentially more iterations when choosing smaller step sizes but every iteration is computationally cheaper in comparison to SD and AccG. Moreover, a small adaptation to step 2 of Algorithm \ref{algo:ACC_GRAD_wo_Q}, where we include a weighting parameter in the $\max$ problem, might allow us to diversify solutions, similar to the weighted sum method.
\begin{figure}
\centering
    \begin{subfigure}[b]{0.32\textwidth}
        \centering
        \includegraphics[width=\linewidth]{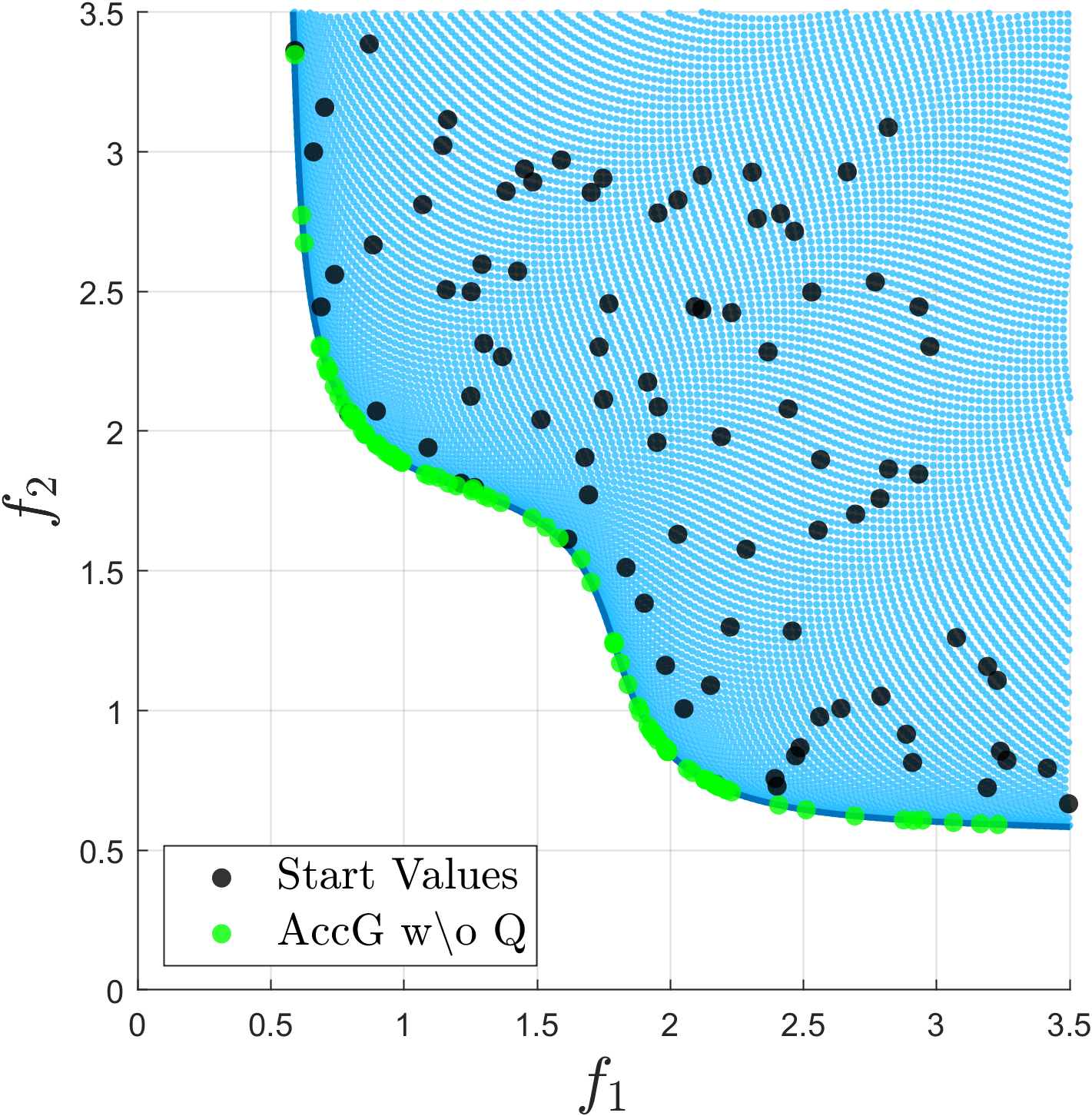}
  \caption{$s = \expnumber{5}{-3}$}\label{fig:ParetoFront_1000_h_5e-3}
    \end{subfigure}
    \hfill
    \begin{subfigure}[b]{0.32\textwidth}
        \centering
        \includegraphics[width=\linewidth]{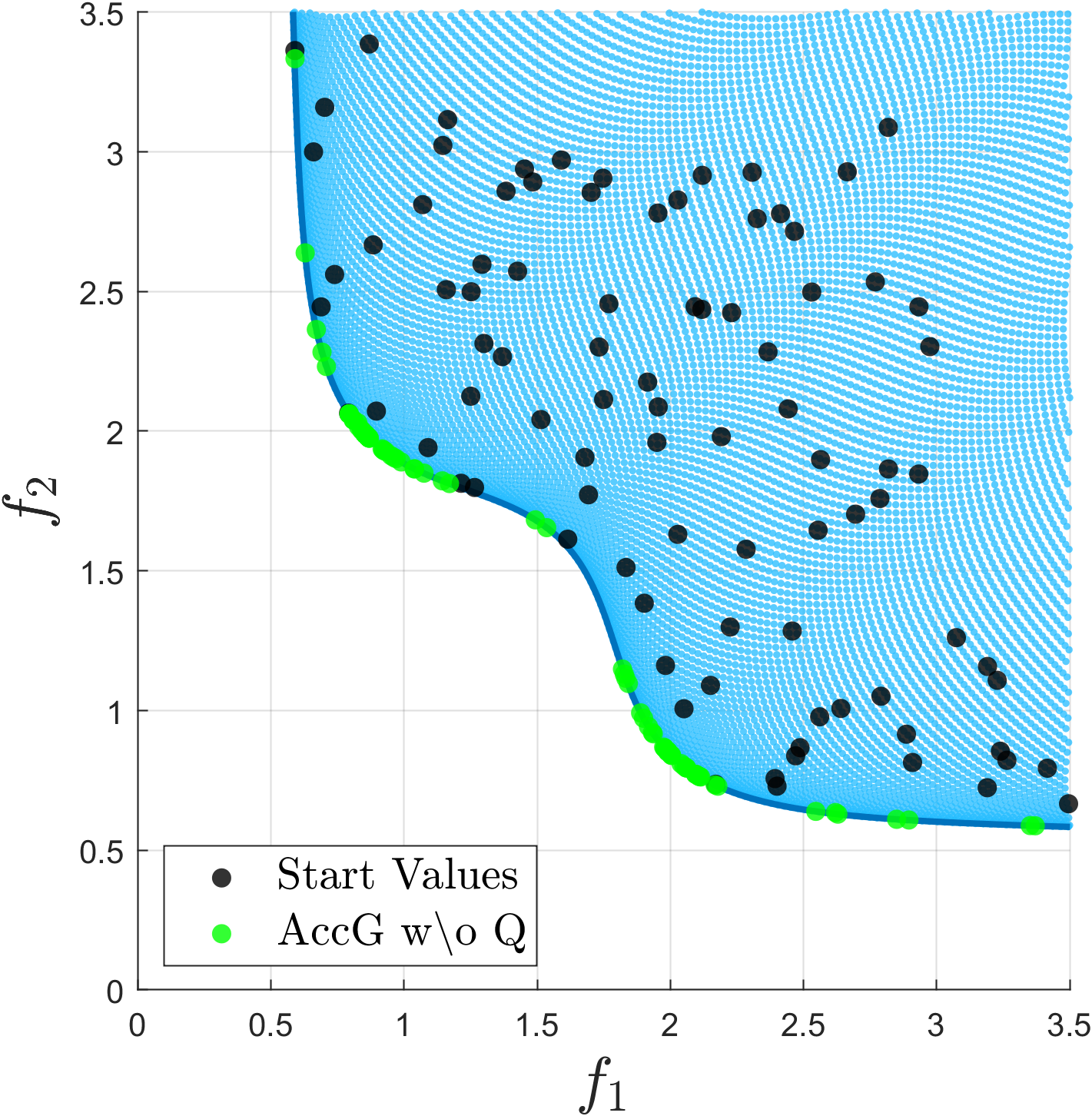}
  \caption{$s = \expnumber{1}{-2}$}\label{fig:ParetoFront_1000_h_1e-2}
    \end{subfigure}
     \begin{subfigure}[b]{0.32\textwidth}
        \centering
        \includegraphics[width=\linewidth]{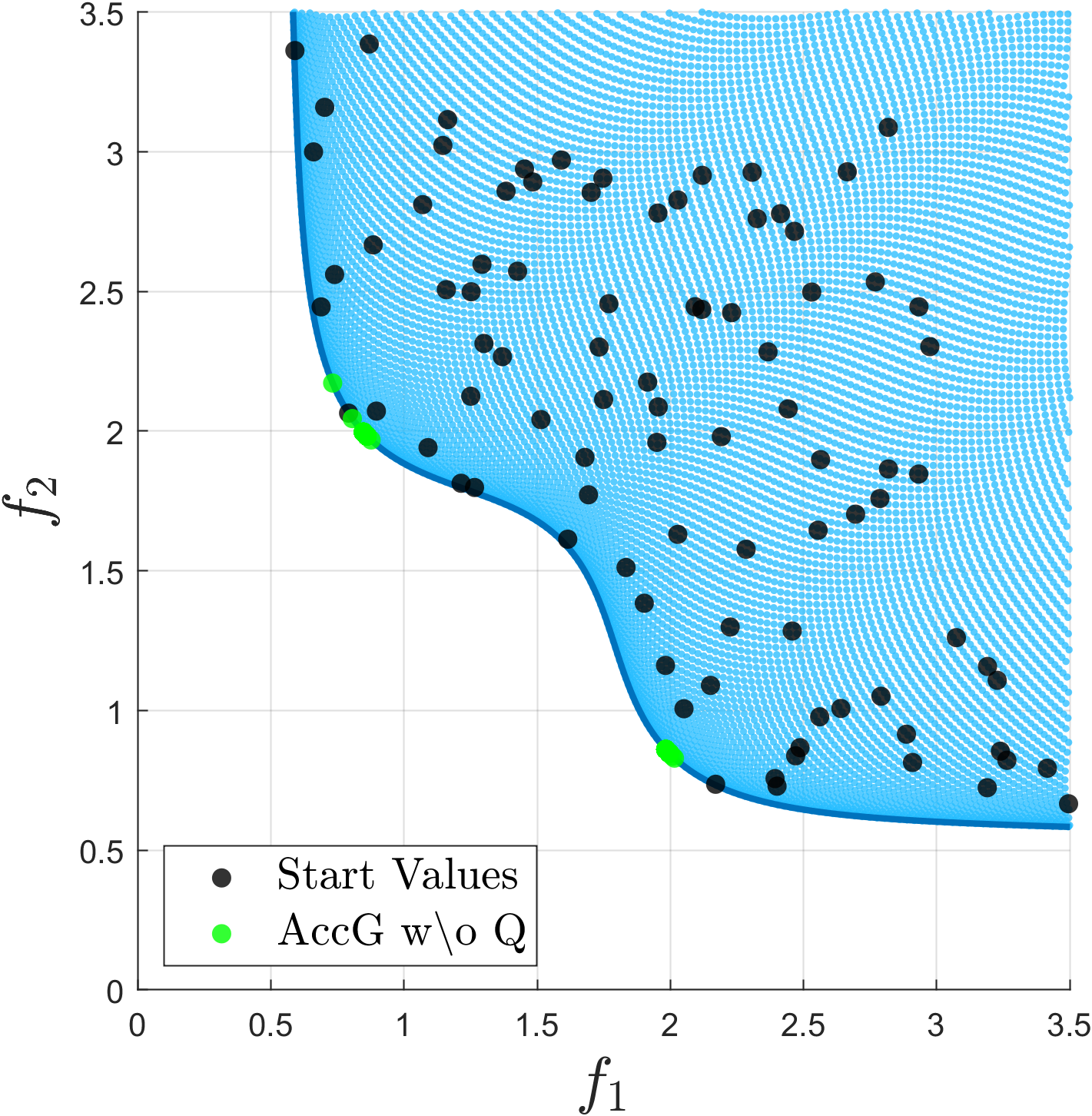}
        \caption{$s = \expnumber{5}{-2}$}\label{fig:ParetoFront_1000_h_5e-2}
    \end{subfigure}
    \caption{Values of the objective functions in the image space for different step sizes $s = \expnumber{5}{-3}, \expnumber{1}{-2}, \expnumber{5}{-2}$ for the algorithm AccG w\textbackslash o Q.}
        \label{fig:pareto_front2}
\end{figure}
\section{Conclusion and Open Questions}
\label{sec:conclusion}
We present the novel inertial gradient-like dynamical system \eqref{eq:IMOG'} for Pareto optimization. We show that trajectories of this system converge weakly to Pareto critical points of \eqref{eq:MOP}. Based on this, we define a novel inertial gradient method for multiobjective optimization and show weak convergence to Pareto critical points. We derive an accelerated gradient method from the informally introduced internal gradient-like system \eqref{eq:MAVD} which incorporates asymptotically vanishing damping. Using the concept of merit functions, we show that our method possesses an improved convergence rate. Using a different discretization of the system \eqref{eq:MAVD}, we define a further accelerated gradient method which does not require the solution to a quadratic optimization problem in every step. A comparison on selected test problems shows that the accelerated methods are in fact superior to the plain multiobjective steepest descent method.

There are a lot of open questions arising from the presented work. The gradient system \eqref{eq:IMOG'} can be analyzed for different problem classes. Also, we have not formally discussed the system \eqref{eq:MAVD}. It would be interesting to see if we can derive solutions from this system for $\alpha > 3$ and obtain stronger convergence results similar to the single objective setting. In addition, we can adapt our gradient systems and algorithms to treat problems with a separable smooth and nonsmooth structure using proximal methods. It would also be interesting to investigate the behavior of our algorithms for high-dimensional and nonconvex problems. In addition, one could apply the presented algorithms in the area of machine learning, e.g., for multitask learning problems \cite{SK18}.
\appendix
\section{Two Lemmas on Convex Projections}
\begin{mylemma}
Let $\H$ be a real Hilbert space, $C \subset \H$ a convex and compact set and $\eta \in \H$ a fixed vector. Then, $\xi \in \H$ is a solution to the problem
\begin{align}
\label{eq:projection_problem1}
\text{Find $\xi \in \H$ such that : }\, \eta = \proj_{C + \xi} (0),
\end{align}
if and only if it has the form $\xi = \eta - \mu$, where $\mu$ is a solution to the constrained optimization problem $\min_{\mu\in C} \,\langle \mu, \eta \rangle$.
\label{lem:proj_1}
\end{mylemma}
\begin{proof}
First, we show that an element of the form $\xi = \eta - \mu$, with $\mu$ a solution to $\min_{\mu\in C} \,\langle \mu, \eta \rangle$ is a solution to problem \eqref{eq:projection_problem1}. The set of minimizers of the problem $\min_{\mu\in C} \,\langle \mu, \eta \rangle$ is nonempty, since $C$ is compact. Let $\mu \in \argmin_{\mu\in C} \langle \mu, \eta \rangle$. Since $C$ is convex, the first order optimality condition for this problem gives that for all $x\in C$ it holds that $\langle x - \mu, \eta \rangle \ge 0$ and hence
\begin{align*}
    \langle x + \xi - (\mu + \xi), \eta \rangle \ge 0.
\end{align*}
Since we have chosen $\xi = \eta - \mu$ the equation above reads as
\begin{align*}
    \langle x + \xi - \eta, \eta \rangle \ge 0,
\end{align*}
which is equivalent to $\eta = \proj_{C + \xi} (0)$. The other direction works analogously. If the vector $\xi$ is a solution to problem \eqref{eq:projection_problem1} this guarantees that $\mu = \xi - \eta$ satisfies the first order optimality condition for problem $\min_{\mu\in C} \,\langle \mu, \eta \rangle$. Since problem $\min_{\mu\in C} \,\langle \mu, \eta \rangle$ is convex and defined over a convex set, this is equivalent to $\mu$ being an optimal solution to $\min_{\mu\in C} \,\langle \mu, \eta \rangle$.
\end{proof}
\begin{mylemma}
Let $\H$ be a real Hilbert space, $C \subset \H$ a convex and closed set and $a > 0, \nu \in \H$ fixed, with $a \neq -1$ . Then, the problem
\begin{align}
\label{eq:inverse_projection3}
    \text{Find $\xi\in\H$ such that :  }\, -a(\xi + \nu) = \proj_{C + \xi}(0),
\end{align}
has the unique solution $\xi = -\left(\frac{1}{1+a}\proj_C\left(\nu\right) + \frac{a}{1+a}\nu\right)$.
\label{lem:proj_2}
\end{mylemma}
\begin{proof}
First, we show that $\xi = -\left(\frac{1}{1+a}\proj_C\left(\nu\right) + \frac{a}{1+a}\nu\right)$ is a solution to \eqref{eq:inverse_projection3}. It is easy to check that $-a(\xi + \nu) \in C + \xi$. Define $p \coloneqq \proj_C\left(\nu\right)$. For all $x \in C$ it holds that  $\langle x - p, p - \nu \rangle \ge 0$ and hence for all $x \in C$ we get
\begin{align*}
    \langle x + \xi + a(\xi + \nu), a(\xi + \nu) \rangle \le 0,
\end{align*}
which is equivalent to
\begin{align*}
     -a(\xi + \nu) = \proj_{C + \xi}(0).
\end{align*}
The uniqueness follows the same way. Assume we have a solution $\Tilde{\xi}$ to \eqref{eq:inverse_projection3}. By the same computations as above it holds that for all $x \in C$
\begin{align*}
    \langle x + (1+a)\Tilde{\xi} + a\nu, \Tilde{\xi} + \nu \rangle \le 0.
\end{align*}
This is equivalent to
\begin{align*}
    -((1+a)\Tilde{\xi} + a\nu) = \proj_C(\nu),
\end{align*}
from which follows that $\xi = \tilde{\xi}$ is the unique solution.
\end{proof}
\bibliographystyle{unsrt}
{\footnotesize
\bibliography{literature}
}
\nocite{*}
\end{document}